\documentclass[12pt]{amsart}
\usepackage{extsizes}
\usepackage[utf8]{inputenc}

\usepackage{amsmath}%
\usepackage{amsfonts}%
\usepackage{amssymb}%
\usepackage{amsthm}
\usepackage{bm}
\usepackage{graphicx}
\usepackage{stmaryrd}
\usepackage{centernot}
\usepackage{url}
\usepackage[english]{babel}
\usepackage[margin=1in]{geometry}
\usepackage[shortlabels]{enumitem}

\theoremstyle{plain}
\newtheorem{theorem}{Theorem}[section]
\newtheorem{corollary}[theorem]{Corollary}
\newtheorem{lemma}[theorem]{Lemma}
\newtheorem{fact}[theorem]{Fact}
\newtheorem{proposition}[theorem]{Proposition}

\newtheorem{claim}{Claim}

\newtheorem*{theorem*}{Theorem}
\newtheorem*{corollary*}{Corollary}
\newtheorem*{lemma*}{Lemma}
\newtheorem*{fact*}{Fact}
\newtheorem*{proposition*}{Proposition}

\theoremstyle{definition}
\newtheorem{definition}[theorem]{Definition}

\newcommand{\countlogic}{\mathcal{L}_{\omega_1, \omega}}

\newcommand{\isom}{\cong}

\newcommand{\mcA}{\mathcal{A}}
\newcommand{\mcB}{\mathcal{B}}
\newcommand{\mcC}{\mathcal{C}}
\newcommand{\mcD}{\mathcal{D}}
\newcommand{\mcF}{\mathcal{F}}
\newcommand{\mcH}{\mathcal{H}}
\newcommand{\mcI}{\mathcal{I}}

\newcommand{\mcK}{\mathcal{K}}
\newcommand{\mcL}{\mathcal{L}}

\newcommand{\mcM}{\mathcal{M}}
\newcommand{\mcN}{\mathcal{N}}

\newcommand{\oc}{\mathbf{c}}

\newcommand{\BPi}{\mathbf{\Pi}}

\newcommand{\inj}{\text{inj}}
\newcommand{\fin}{\text{fin}}

\DeclareMathOperator{\Aut}{Aut}
\DeclareMathOperator{\supp}{supp}
\DeclareMathOperator{\cl}{cl}
\DeclareMathOperator{\clmin}{Kcl}

\DeclareMathOperator{\Stab}{Stab}
\DeclareMathOperator{\Drk}{Drk}
\DeclareMathOperator{\Dcl}{Dcl}

\DeclareMathOperator{\Krk}{Krk}

\newcommand{\forkindep}[1][]{%
  \mathrel{
    \mathop{
      \vcenter{
        \hbox{\oalign{\noalign{\kern-.3ex}\hfil$\vert$\hfil\cr
              \noalign{\kern-.7ex}
              $\smile$\cr\noalign{\kern-.3ex}}}
      }
    }\displaylimits_{#1}
  }
}

\begin{document}
\title{Classification strength of Polish groups: involving $S_\infty$}
\date{\today}
\author{Shaun Allison}
\address{Department of Mathematics, University of Toronto}

\subjclass[2020]{Primary 54H05, 37B02, 54H11; Secondary 03E15, 03C15}

\keywords{Polish group, Borel reduction, Knight's model, countable model, Fra\"{i}ss\'{e} limit}

\begin{abstract}
    In recent years, much work has been done to measure and compare the complexity of orbit equivalence relations, especially for certain classes of Polish groups.
    We start by introducing some language to organize this previous work, namely the notion of \textbf{classification strength} of Polish groups.
    Broadly speaking, a Polish group $G$ has stronger classification strength than $H$ if every orbit equivalence relation induced by a continuous action of $H$ on a Polish space can be ``emulated" by such an action of $G$ in the sense of Borel reduction.

    Among the non-Archimedean Polish groups, the groups with the highest classification strength are those that involve $S_\infty$, the Polish group of permutations of a countably-infinite set.
    We prove that several properties, including a weakening of the disjoint amalgamation in Fra\"{i}ss\'{e} theory, a weakening of the existence of an absolute set of generating indiscernibles, and not having ordinal rank in a particular coanalytic rank function, are all equivalent to a non-Archimedean Polish group involving $S_\infty$.
    Furthermore, we show the equivalence relation $=^+$, which is a relatively simple benchmark equivalence relation in the theory of Borel reducibility, can only be classified by such groups that involve $S_\infty$.
\end{abstract}

\maketitle

\section{Introduction}
Invariant descriptive set theory is concerned with definable equivalence relations and definable reductions between them.
In particular, we usually consider equivalence relations living on Polish spaces, where the equivalence relations are \emph{analytic}, i.e. analytic as a subset of the product space.
The definable reductions that we consider are usually the \emph{Borel} ones.
Given analytic equivalence relations $E$ and $F$ living on Polish spaces $X$ and $Y$ respectively, a \emph{Borel reduction} from $E$ to $F$ is a Borel function $f : X \rightarrow Y$ satisfying $x \mathrel{E} x'$ iff $f(x) \mathrel{F} f(x')$ for every $x, x' \in X$.
When such a reduction exists, we say that $E$ is \emph{Borel-reducible} to $F$, i.e. $E \le_B F$.
In the case that $E$ and $F$ represent classification problems in math, such as isomorphism of graphs or conjugacy of measure-preserving translations, then from $E \le_B F$ we can conclude that the problem represented by $E$ is no harder than the one represented by $F$.

One common way to show that $E$ does \emph{not} Borel reduce to $F$ is by showing that $E$ is \emph{generically-ergodic} with respect to $F$.
A function $f : X \rightarrow Y$ is a homomorphism from $E$ to $F$ iff for every $x, y \in X$, if $x \mathrel{E} y$ then $f(x) \mathrel{F} f(y)$.
Then we say that $E$ is generically-ergodic with respect to $F$ iff for every Borel homomorphism $f : X \rightarrow Y$ from $E$ to $F$, there is a comeager $C \subseteq X$ such that for any $x, y \in C$ we have $f(x) \mathrel{F} f(y)$.
In the case that $E$ does not have a comeager class, this precludes the existence of a Borel reduction to $F$.

While Borel reducibility yields a relational measure of complexity between different equivalence relations, there are also absolute measures of complexity that place equivalence relations in a neatly-organized ``hierarchy''.
One very successful program along these lines has been, in the case of Borel equivalence relations, to consider their place in the Borel hierarchy via a notion of \emph{potential Borel complexity}.
This was essentially initiated in the seminal \cite{HKL1998}, where a connection to potential Borel complexity was made to the set-theoretic complexity of any definable assignment of invariants.
In particular, in the large class of equivalence relations that they consider, an equivalence relation is potentially $\BPi^0_2$ iff there is a Borel assignment of reals as invariants, potentially $\BPi^0_3$ iff there is a Borel assignment of sets of reals as invariants, potentially $\BPi^0_4$ iff there is a Borel assignment of sets of sets of reals as invariants, and so on.

The case of \textbf{orbit equivalence relations} gives another opportunity to produce a meaningful hierarchy of complexity, by instead studying the acting (``classifying") groups.
Many definable equivalence relations that arise in practice are actually orbit equivalence relations.
Given a Polish group $G$, a \emph{Polish $G$-space} is a Polish space $X$ along with a continuous action $G \curvearrowright X$.
We use $E^G_X$ to denote the induced orbit equivalence relation, which is analytic, and moreover has the property that every orbit is Borel, see \cite[2.3.4]{BeckerKechris1996}.

Formally, given an equivalence relation $E$ on a Polish space $X$, we say that a Polish group $G$ \textbf{classifies} $E$ iff there is a Polish $G$-space $Y$ such that $E \le_B E^G_Y$.
This notion can be used to separate equivalence relations up to Borel reduction.
We identify three examples of this phenomenon of particular importance.

For the first example, we say that an equivalence relation $E$ is \emph{classifiable by countable structures} iff it is classifiable by $S_\infty$, the Polish group of automorphisms of a countably-infinite set with the natural topology.

An important benchmark equivalence relation is $=^+$, the Friedman-Stanley jump of equality, which has particular importance in the study of potential Borel complexity in \cite{HKL1998}.
It is defined on $\mathbb{R}^\omega$ by 
\[(x_n) \mathrel{=^+} (y_n) \quad \text{iff} \quad \{x_n \mid n \in \omega\} = \{y_n \mid n \in \omega\}.\]
This is easily seen to be classifiable by countable structures.

Hjorth showed that any orbit equivalence relation that is generically-ergodic with respect to $=^+$ is also generically-ergodic with respect to any action of $S_\infty$.
If the orbit equivalence relation moreover does not have a comeager orbit then it is not classifiable by countable structures.
Hjorth also isolated a dynamical property of some orbit equivalence relations called \emph{turbulence}, which precludes them from being classifiable by countable structures (see \cite{HjorthBook} and \cite{Hjorth2002Turbulence}).

For the second example, we observe that further separations can be found among the orbit equivalence relations classifiable by countable structures.
Recall that a Polish group is \emph{complete left-invariant} (or cli) iff it has a complete metric $d_G$ which is compatible in the sense that it generates the topology, and furthermore is left-invariant, i.e.
\[d_G(g, g') = d_G(hg, hg')\]
for every group elements $g, g', h$.
In \cite{Hjorth1999}, Hjorth identifies a metamathematical property of equivalence relations that precludes them from being classifiable by cli Polish groups, and shows that $=^+$ exhibits this property.
This property was further explored and dubbed ``unpinned" in \cite{Kanovei2008}.

For the third example, we recall a stronger invariance notion than a Polish group being cli is being \emph{two-sided invariant} (or tsi).
A Polish group is tsi iff there is a compatible complete metric $d_G$ satisfying
\[ d_G(g, g') = d_G(hg, hg') = d_G(gh, g'h)\]
for every group elements $g, g', h$.

A natural example of an equivalence relation classifiable by a tsi Polish group is $E_\infty$, defined as the orbit equivalence relation induced by the Bernoulli-shift action of $F_2 \curvearrowright \mathbb{R}^{F_2}$, where $F_2$ is the free group on two generators with the discrete topology.
Since $F_2$ is a countable discrete group, it is tsi for trivial reasons.

In \cite{AllisonTSI}, it was shown that an orbit equivalence relation that is generically ergodic with respect to $E_\infty$ is generically ergodic with respect to any orbit equivalence relation induced by a tsi non-Archimedean Polish group.
This strongly parallels Hjorth's result relating generic ergodicity with respect to $=^+$ and orbit equivalence relations induced by any non-Archimedean Polish group.

In \cite{CC2022} Clemens-Coskey introduced an equivalence relation $E_0^{[\mathbb{Z}]}$, called the $\mathbb{Z}$-jump of $E_0$, which is classifiable by a cli Polish group.
They asked if it is in fact classifiable by a tsi Polish group.
However they showed that it is generically ergodic with respect to $E_\infty$ and furthermore has all meager classes, thus by \cite{AllisonTSI} it is not classifiable by a tsi non-Archimedean Polish group.

This was generalized by the author and Panagiotopoulos in \cite{AllisonPanagio2021} to general tsi Polish groups, and a purely dynamical property was identified similar to Hjorth's notion of turbulence which serves as an obstruction to classifiability by tsi Polish groups.
In particular, $E_0^{[\mathbb{Z}]}$ is not classifiable by \emph{any} tsi Polish group.
We will not need the definition of $E_0^{[\mathbb{Z}]}$ in this work, but its definition, as well as an exploration of an interesting hierarchy of similar equivalence relations, can be found in \cite{CC2022}.

We can observe that by considering properties of the classifying group we can now start to see that it produces a meaningful hierarchy of equivalence relations which are classifiable by countable structures:
\begin{align*}
    &\textbf{classifiable by non-Archimedean TSI}\\
    \subsetneq \; &\textbf{classifiable by non-Archimedean CLI} \\
    \subsetneq \;&\textbf{classifiable by } \mathbf{S_\infty}.
\end{align*}
In another upcoming paper \cite{AllisonPanagio2026}, we are showing with Panagiotopoulos that there is a finer hierarchy below classifiability by CLI, and in this paper we expose another hierarchy which exists above classifiability by CLI.
With this picture in mind, the following definition seems natural:

\begin{definition}
    Say that $G$ is \textbf{stronger in classification strength} than $H$ iff for every Polish $H$-space $X$, the orbit equivalence relation $E^H_X$ is classifiable by $G$.
\end{definition}

The following weak restatement of a result of Mackey and Hjorth further motivates this definition.
Recall that $G$ \emph{involves} $H$ iff there is a closed subgroup $G_0$ of $G$ and a continuous surjective homomorphism from $G_0$ onto $H$. Note that this is also sometimes said ``$H$ divides $G$".

\begin{lemma}[Mackey, Hjorth]\label{lem:mackey}
    If $G$ involves $H$, then $G$ is stronger in classification strength than $H$.
\end{lemma}

The non-Archimedean Polish groups are exactly those that are isomorphic to closed subgroups of $S_\infty$.
Thus the non-Archimedean Polish groups which involve $S_\infty$ have maximal classification strength among the non-Archimedean Polish groups.
A result of Hjorth implies the converse \cite{Hjorth2001} (see also the more recent \cite{LU2021} which recovers this result using a different strategy.)

Hjorth's result gives a metamathematical sufficient condition for a non-Archimedean Polish group to involve $S_\infty$.
On the other hand, a paper of Baldwin-Friedman-Koerwien-Laskowski contains an argument that the automorphism group of the limit of any Fra\"{i}ss\'{e} class which satisfies disjoint amalgamation involves $S_\infty$.
However, not much else was known, and there had not yet been any comprehensive study of the division between the Polish groups which do and do not involve $S_\infty$.

In this paper, we identify several seemingly unconnected properties of non-Archimedean Polish groups, which are equivalent to involving $S_\infty$.
This tells us, we believe, that such groups have an interesting and deep structure and are worthy of further study.

Our main result is the following, where the terms mentioned in equivalences (2), (3), (4), and (5) are yet to be defined.

\begin{theorem}
Let $G = \Aut(\mcM)$ be a non-Archimedean Polish group for some countable structure $\mcM$ in a countable language $\mcL$. Then the following are equivalent:
\begin{enumerate}
    \item $G$ does not involve $S_\infty$;
    \item Every disjointifying closure operator on $M$ is trivial;
    \item $\Krk(\mcM) < \infty$;
    \item $\Krk(\mcM) < \omega_1$;
    \item There is no indiscernible support function on $\mcM$;
    \item $G$ does not classify $=^+$;
\end{enumerate}
\end{theorem}

This is the beginning of a larger project to understand the hierarchy of Polish groups not involving $S_\infty$ which arises from the rank function $\Krk$.

\subsection{Acknowledgements}
The author would like to thank Omer Ben-Neria, Clinton Conley, Aristotelis Panagiotopoulos, and Spencer Unger for many valuable conversations. This work was partially funded by the Israel Science Foundation (grant no. 1832/19).

\section{Preliminaries}

\subsection{Countable model theory}

We briefly review a few basic concepts and notation from the model theory of countable structures.
We will always use $\mcL$ to refer to a countable relational language.
Given an $\mcL$-structure $\mcM$, we will write $M$ to refer to the underlying set of $\mcM$.
If $a, a' \in M$ and $B \subseteq M$, we write $a \isom_B a'$ iff there is an automorphism $\pi$ of $\mcM$ satisfying $\pi(a) = a'$ and $\pi(b) = b$ for every $b \in B$.

We write $\mcN \preceq_{\mcL} \mcM$ to denote that $\mcN$ is an $\mcL$-substructure of $\mcM$, and we write $\mcN \preceq_{\mcL_{\omega_1, \omega}} \mcM$ iff $\mcN$ is an $\mcL_{\omega_1, \omega}$-elementary substructure of $\mcM$.
We write $\prec_{\mcL}$ for proper $\mcL$-substructure and respectively for $\prec_{\mcL_{\omega_1, \omega}}$.
The following characterization of $\mcL_{\omega_1, \omega}$-elementary substructure which is the ``correct'' analogue of the Tarski-Vaught test, is extremely useful yet seems to be neglected in the literature, so a proof is included.

\begin{lemma}\label{lem:elementary_substructure}
For any two countable $\mcL$-structures $\mcN$ and $\mcM$, we have $\mcN \prec_{\mcL_{\omega_1, \omega}} \mcM$ iff $\mcN$ is an $\mcL$-substructure of $\mcM$ and for every finite $B \subseteq N$ and $a \in M$, there is some $a' \in N$ such that $a' \isom_B a$.
\end{lemma}

\begin{proof}
We start with the forward direction, assuming $\mcN$ is an $\mcL_{\omega_1, \omega}$-elementary substructure of $\mcM$, and fixing finite $B \subseteq N$ and $a \in M$.
List the elements of $B$ as $b_1, \dots, b_n$ and let $\varphi(x_1, \dots, x_n, y)$ be the Scott sentence of $(\mcM, b_1, \dots, b_n, a)$.
For a review of the Scott analysis, please see e.g. \cite[Chapter 12]{Gao2008}.
Then $\mcM \models \exists y, \; \varphi(b_1, \dots, b_n, y)$ and so by $\mcL_{\omega_1, \omega}$-elementarity, we have $\mcN \models \exists y, \; \varphi(b_1, \dots, b_n, y)$.
Fix $a' \in N$ such that $\mcN \models \varphi(b_1, \dots, b_n, a')$.
Again by elementarity we have $\mcM \models \varphi(b_1, \dots, b_n, a')$.
As $\varphi$ is a Scott sentence and $\mcM$ is countable there is an automorphism $\pi$ of $\mcM$ such that $\pi(b_i) = b_i$ for all $i$ and $\pi(a) = a'$, as desired.

For the backwards direction, assume $\mcN$ is an $\mcL$-substructure of $\mcM$ and for every finite $B \subseteq N$ and $a\in M$ there is $a' \in N$ such that $a' \isom_B a$.
Let $\mcH$ be the set of all formulas $\varphi(x_1, \dots x_n)$ in $\mcL_{\omega_1, \omega}$ such that for every $a_1, \dots, a_n \in N$ we have $\mcN \models \varphi(a_1, \dots, a_n)$ iff $\mcM \models \varphi(a_1, \dots, a_n)$.
This set $\mcH$ contains all atomic $\mcL$-formulas since $\mcN$ is an $\mcL$-substructure of $\mcM$.
Moreover, it is clearly closed under infinite disjunction and conjunction and negation.
It suffices to show that it is closed under existential quantification, as closure under universal quantification will also follow immediately and this would imply that $\mcH$ contains all the $\mcL_{\omega_1, \omega}$ formulas.
Indeed, suppose $\varphi(x_1, \dots, x_n, y)$ is in $\mcH$ and fix any $a_1, \dots, a_n \in N$.
Then $\mcN \models \exists y, \varphi(a_1, \dots, a_n, y)$ if and only if there exists some $b \in N$ such that $\mcN \models \varphi(a_1, \dots, a_n, b)$ if and only if (by the fact that $\varphi$ is in $\mcH$) there exists some $b \in N$ such that $\mcM \models \varphi(a_1, \dots, a_n, b)$ if and only if (by the assumption) there exists some $b \in M$ such that $\mcM \models \varphi(a_1, \dots, a_n, b)$ if and only if $\mcM \models \exists y, \; \varphi(a_1, \dots, a_n, y)$.
\end{proof}

We say $\mcM$ is \emph{ultrahomogeneous} iff for any two tuples $\bar{a}$ and $\bar{b}$ from $M$, if they have the same quantifier-free type, then there is an automorphism sending one to the other.

%We will occasionally make reference to the \emph{Scott sentence} of a structure, though we will not make use of the theory of such sentences in any deep sense.
%It is important to know that for every $\mcL$-structure $\mcM$, there is a canonical $\mcL_{\omega_1, \omega}$ sentence $\varphi_{\mcM}$ (called the Scott sentence of $\mcM$) such that for any other $\mcL$-structure $\mcN$, the following are equivalent
%\begin{enumerate}
%    \item $\mcN \models \varphi_{\mcM}$;
%    \item $\varphi_{\mcN} = \varphi_{\mcM}$;
%    \item $\mcM \equiv_{\mathcal{L}_{\omega_1, \omega}} \mcN$.
%\end{enumerate}
%In the case that both $\mcM$ and $\mcN$ are countable, we also have the additional equivalence:
%\begin{enumerate}
%    \item[(4)] $\mcM \isom \mcN$. 
%\end{enumerate}
%This is enough background to make sense of any situation in which we bring up the Scott sentence, but one can find a development of Scott sentences and their definition in \cite{Gao2008}.

\subsection{Non-Archimedean Polish groups}

A Polish group is called \emph{non-Archimedean} if it has a countable local basis of the identity of open subgroups.
The most important such group is $S_\infty$, the Polish group of permutations of a countably-infinite set (with the discrete topology), equipped with the pointwise convergence topology.
There are several useful equivalent formulations of a Polish group being non-Archimedean:

\begin{fact}\label{fact:nonarchimedean}
Let $G$ be a Polish group. The following are equivalent:
\begin{enumerate}
    \item $G$ is non-Archimedean;
    \item $G$ is isomorphic to $\Aut(\mcM)$ for a countable structure $\mcM$ in a countable language;
    \item $G$ is isomorphic to a closed subgroup of $S_\infty$;
    \item $G$ has a compatible left-invariant ultrametric;
    \item $G$ is involved by $S_\infty$.
\end{enumerate}
\end{fact}

\begin{proof}
See \cite[Theorem 2.4.1, Theorem 2.4.4]{Gao2008} for the equivalence of (1), (2), (3), and (4).
Statement (3) immediately implies (5).
We show that (5) implies (1).
Indeed, if $G$ is involved by $S_\infty$ then there is a closed subgroup $H \le S_\infty$ and a continuous surjective homomorphism from $H$ onto $G$.
By the equivalence of (3) and (1), we can fix a countable local basis of the identity of $H$ of open subgroups of $H$.
The images of these open subgroups are open in $G$ by the open mapping theorem, see e.g. \cite[Theorem 2.3.3]{Gao2008}.
These sets easily form a countable local basis of the identity of $G$ of open subgroups of $G$.
\end{proof}

We will make use of equivalence (2) the most, as it allows the use of language and techniques from model theory, even though our use of model theory will not be particularly deep.
We can assume that $\mcM$ is ultrahomogeneous by adding new relations without changing the set of automorphisms.
This is done as follows:
for each tuple $\bar{b}$ from $M$, add a new relation $R_{\bar{b}}$ to the language and declare $R_{\bar{b}}^\mcM(\bar{a})$ iff there is an automorphism of $\mcM$ sending $\bar{a}$ to $\bar{b}$.

\subsection{Non-Archimedean Polish groups that are cli and not cli}

Recall that a Polish group is \emph{cli} iff it has a complete left-invariant metric which is compatible with the topology.
In the case of a non-Archimedean Polish group, there is a nice characterization.

\begin{fact}[Gao, \cite{Gao1998}]
Let $G = \Aut(\mcM)$ be a non-Archimedean Polish group. Then the following are equivalent:
\begin{enumerate}
    \item $G$ is cli
    \item $\mcM$ has no nontrivial $\mcL_{\omega_1, \omega}$-elementary substructure;
    \item there is no uncountable model of the Scott sentence of $\mcM$
\end{enumerate}
\end{fact}
There is also a nice rank function due to Deissler which characterizes when $G$ is cli, which we will define and discuss later.

A Polish group is \emph{tsi} iff it has a complete two-sided invariant metric which is compatible with the topology.
An example of a Polish group which is cli but not tsi is the automorphism group of the linear order with order-type $\mathbb{Z} * \mathbb{Z}$, which is just the lex order on $\mathbb{Z} \times \mathbb{Z}$.
A natural action of this group induces the equivalence relation $E_0^{[\mathbb{Z}]}$ mentioned in the introduction, discussed further in \cite{CC2022}, \cite{AllisonTSI}, \cite{AllisonPanagio2021}.

The most important non-cli Polish group from our perspective is the automorphism group of Knight's model.
It is significant in the sense that it was the simplest known Polish group which is not cli but does not involve $S_\infty$.
Knight's model $\mcK$ is a countable structure in the language $\mcL = \{<, f_n\}_{n \in \omega}$, where $<$ is a binary relation and each $f_n$ is a unary function which satisfies:
\begin{enumerate}
    \item $<$ is a linear order on $K$;
    \item for every $a \in K$, $\{b \in K \mid b < a\} = \{f_n(a) \mid n \in \omega\}$;
    \item $\mcK$ is 1-transitive (i.e. there is an automorphism between any two elements);
    \item there is a non-trivial $\mcL_{\omega_1, \omega}$-elementary substructure of $\mcK$.
\end{enumerate}

By Gao's characterization, (4) implies that $\Aut(\mcK)$ is not cli.
A concept of \cite{Hjorth1999} which was further developed by Kanovei under the name of ``pinned'' and ``unpinned'' equivalence relations in \cite{Kanovei2008} and Larson and Zapletal in \cite{LarsonZapletal2020} shows that $\Aut(\mcK)$ does not classify $=^+$ and thus does not involve $S_\infty$.
This uses metamathematical techniques (namely, forcing).
We recover the same result in this paper using elementary techniques.

\subsection{Examples of non-Archimedean Polish groups involving $S_\infty$}

The original intention of Knight's model was to produce a countable structure whose Scott sentence has a model of cardinality $\aleph_1$, but no larger.
It is said that the Scott sentence of Knight's model ``characterizes" $\aleph_1$.
In \cite{Hjorth2002Knight}, Hjorth constructed, for each $\alpha$, a countable structure characterizing $\aleph_\alpha$.
Rather than a generalization of Knight's model, Hjorth's construction relies on a very different construction method which resulted in structures whose automorphism groups involve $S_\infty$.

Hjorth's model characterizing $\aleph_1$, which we denote by $\mcH$, is a countable ultrahomogeneous structure in a language with binary relations $S_n$ for every $n$, and $k+2$-ary relations $R_k$ for every $k$, satisfying
\begin{enumerate}
    \item there is some function $f : [H]^2 \rightarrow \omega$ such that for every $a$ and $b$ in $H$ and $n \in \omega$, we have $S_n^{\mcH}(a, b)$ iff $f(\{a, b\}) = n$;
    \item for every $a$ and $b$, the set $S(\{a, b\})$ of all $c$ such that $f(\{a, c\}) = f(\{b, c\})$ is finite, and $R_k^{\mcH}(a, b, \bar{c})$ iff $\bar{c}$ enumerates $S(\{a, b\})$.
\end{enumerate}
It's straightforward to check that every proper elementary $\mcL_{\omega_1, \omega}$-substructure of any model of the Scott sentence of $\mcH$ must be countable.
Thus there are no models of the Scott sentence of $\mcH$ of cardinality $\aleph_2$ or higher.
On the other hand, Hjorth showed that the automorphism group of $\mcH$ involves $S_\infty$ (in the process of showing that its Scott sentence has an uncountable model).
This is done by exploiting the fact that $\mcH$ satisfies disjoint amalgamation.
The fact that the automorphism groups of limits of Fra\"{i}ss\'{e} structures satisfying disjoint amalgamation must involve $S_\infty$ is stated and proved more explicitly in \cite{BFKL2016}.
Another proof of this, in even greater detail, appeared recently in \cite{LU2022}.
We briefly sketch this argument in Section \ref{sec:closure}, as we will build on this idea.

A classic example of an automorphism group that involves $S_\infty$ is $\Aut(\mathbb{Q}, <)$, the automorphism group of the rational linear order.
To see that it involves $S_\infty$, first partition $\mathbb{Q}$ into countably-many dense subsets $A_n$.
Let $H \le \Aut(\mathbb{Q})$ be the closed subgroup of automorphisms $\pi$ satisfying that for every $n, m \in \omega$ and $a, b \in \mathbb{Q}$, if $a \in A_n$ and $\pi(a) \in A_m$ then $b \in A_n$ iff $\pi(b) \in A_m$.
Then define a continuous homomorphism $f : H \rightarrow S_\infty$ where $f(\pi)$ is defined to be the unique $\sigma \in S_\infty$ such that $a \in A_m$ iff $\pi(a) \in A_{\sigma(m)}$ for every $a$ and $m$.
To see that $f$ is indeed surjective, apply a back-and-forth argument.

We will describe one more example, which we will actually make use of later.
Let $\Delta$ be some countably-infinite group, and consider a free action $\Delta \curvearrowright I$ on a countably-infinite set $I$ with infinitely-many $\Delta$-orbits.
Now consider the closed subgroup $P \le S_I$ of permutations $\pi$ of $I$ satisfying $\pi(\delta \cdot x) = \delta \cdot \pi(x)$ for every $x \in I$.
This group involves $S_\infty$.
To see this, fix a transversal $T \subseteq I$, i.e. a set which intersects every $\Delta$-orbit exactly once.
Since $T$ is a countably-infinite set, the Polish group $S_T$ is isomorphic to $S_\infty$.
Moreover, the map $f : P \rightarrow S_T$ sending each $\pi \in P$ to the unique $\sigma \in S_T$ satisfying $\pi(x) \in \Delta \cdot \sigma(x)$ for every $x \in T$, is easily seen to be a continuous surjective homomorphism.

\subsection{Baire-measurable homomorphisms and the orbit continuity lemma}

Given a Polish group $G$, recall that a Polish $G$-space is a Polish space $X$ along with a continuous action $G \curvearrowright X$. We write the induced orbit equivalence relation as $E^G_X$. Given two such orbit equivalence relations $E^G_X$ and $E^H_Y$, a homomorphism is a function $f : X \rightarrow Y$ such that for every $x, y \in X$, $x \mathrel{E^G_X} y$ implies $f(x) \mathrel{E^H_Y} f(y)$. 

The following so-called ``orbit continuity lemma", due to Hjorth and Lupini-Panagiotopoulos, is central. 
This is the main way to extract information from the existence of Baire-measurable homomorphisms that don't trivialize on a comeager set.
Our statement differs slightly from the statement in \cite{LupiniPanagio2018}, so we give an argument for how to recover this one from theirs.

\begin{lemma}[Hjorth, Lupini-Panagiotopoulos]\label{lem:orbit_continuity}
Suppose $f : X \rightarrow Y$ is a Baire-measurable homomorphism from $E^G_X$ to $E^H_Y$ and $G_0$ is a countable dense subgroup of $G$. Then there is a comeager subset $C \subseteq X$ such that
\begin{enumerate}
    \item $f$ is continuous on $C$;
    \item for every $x \in C$, the set of $g$ such that $g \cdot x \in C$ is comeager;
    \item for every $x \in C$ and for every $g \in G_0$, we have $g \cdot x \in C$;
    \item for every $x_0 \in C$ and every open neighborhood $V \subseteq H$ of the identity, there is a neighborhood $U$ of $x_0$ and a nonempty open neighborhood $W \subseteq G$ of the identity such that for any $x \in C \cap U$ and for every $w \in W$ with $w \cdot x \in C \cap U$, we have $f(w \cdot x) \in V \cdot f(x)$; and
    \item for every $x_0 \in C$ and $g_0 \in G_0$ and nonempty open $W \subseteq H$ such that $f(g_0 \cdot x_0) \in W \cdot f(x_0)$, there is an open neighborhood $U$ of $x_0$ such that for every $x \in C \cap U$, we have $f(g_0 \cdot x) \in W \cdot f(x)$;
\end{enumerate}
\end{lemma}

Note that in (5) the open set $W$ need not contain the identity.

\begin{proof}
Let $C_0$ be a comeager set satisfying the first three conditions as in \cite[Lemma 2.5]{LupiniPanagio2018}. 
In particular, $C_0$ satisfies (1), (2) and
\begin{enumerate}
\item[(4*)]  for every $x_0 \in C_0$ and every nonempty open neighborhood $V \subseteq H$ of the identity, there is a neighborhood $U_0$ of $x_0$ and a nonempty open neighborhood $W \subseteq G$ of the identity such that for any $x \in C_0 \cap U_0$ and for a comeager set of $w \in W$, we have $f(w \cdot x) \in V \cdot f(x)$ and $g \cdot x \in C_0$.
\end{enumerate}
For every $g \in G_0$ and basic open $W$, let $C_{g, W}$ be a comeager set on which the Baire-measurable function $\hat{f} : X \rightarrow 2$ defined by
\[\hat{f}(x) = 1 \leftrightarrow f(g \cdot x) \in W \cdot f(x)\]
is continuous.
Let $C_1$ be the intersection of $C_0$ and each $C_{g, W}$.
Let $C_2$ be the intersection of each $g \cdot C_1$ for $g \in G_0$.
Define
\[C = \{x \in C_2 \mid \forall^*g \in G, \; g \cdot x \in C_2\}.\]
Since $G_0$ contains the identity, we have $C \subseteq C_2 \subseteq C_1 \subseteq C_0$.
Note that $C_1$ and $C_2$ are comeager by continuity of the group action and the countable closure of the meager ideal, and $C$ is comeager by Kuratowski-Ulam.
We argue that this works.

Clause (1) is satisfied because $C \subseteq C_0$.
Clause (2) is satisfied because if $x \in C$ then there is a comeager set $C_G \subseteq G$ such that $g \cdot x \in C_2$ for every $g \in C_G$. Then in fact for any $g \in G$ we have $g \cdot x \in C$ because $C_G g^{-1}$ is comeager in $G$.
Clause (3) is satisfied by the same argument and the fact that $C_2$ is invariant under $G_0$.

Clause (4) is a modification of (4*), with the difference being that we quantify over \emph{all} $w \in W$, not just a set comeager in $W$, such that $w \cdot x \in U_0 \cap C$.
Let $x_0 \in C$ and $V \subseteq H$ be an open neighborhood of the identity of $H$.
Let $\hat{V} \subseteq H$ be a symmetric open neighborhood of the identity of $H$ such that $\hat{V}^2 \subseteq V$.
By (4*), there is an open neighborhood $U \ni x_0$ and an open neighborhood $W \subseteq G$ of the identity of $G$ such that for every $x \in C \cap U$ there is a comeager set of $w \in W$ such that $f(w \cdot x) \in \hat{V} \cdot f(x)$.

Let $x \in U \cap C$ and $w \in W$ such that $w \cdot x \in U \cap C$.
Let $D_0$ be the set of $w' \in W$ such that $f(w' \cdot x) \in \hat{V} \cdot f(x)$ and let $D_1$ be the set of $w' \in W$ such that $f(w'w \cdot x) \in \hat{V} \cdot f(w \cdot x)$.
The set $D_0$ is comeager in $W$.
The set $D_1$ is also comeager in $W$ and thus $D_1w$ is comeager in $Ww$, and since $W \cap Ww$ is nonempty open, we may fix some $w' \in D_0 \cap D_1w$.
Then we have that $f(w' \cdot x) \in \hat{V} \cdot f(x)$ and $f(w'w^{-1} \cdot (w \cdot x)) \in \hat{V} \cdot f(w \cdot x)$.
Thus $f(w \cdot x) \in \hat{V}^{-1}\hat{V} \cdot f(x) \subseteq V \cdot f(x)$ as desired.

Clause (5) is satisfied because for any $x_0 \in C$ and $g \in G_0$ and nonempty open $W \subseteq H$ such that $f(g \cdot x_0) \in W \cdot f(x_0)$ then we have $x_0 \in C_{g, W}$ and so there is an open $U \ni x_0$ such that for every $x \in C_{g, W} \cap U$ we have $f(g \cdot x) \in W \cdot f(x)$. This works.
\end{proof}

\section{A weakening of disjoint amalgamation}\label{sec:closure}

In \cite{Hjorth2002Knight}, and more explicitly in \cite{BFKL2016}, disjoint amalgamation is identified as a sufficient condition for the automorphism group of a structure to involve $S_\infty$.
In this section, we introduce the appropriate weakening of disjoint amalgamation.
We first give a brief review of the Fra\"{i}ss\'{e} theory of classes of finite structures, and review the ideas of \cite{BFKL2016}.
Then we introduce the weakening of disjoint amalgamation and prove that it is necessary and sufficient.

\subsection{A review of Fra\"{i}ss\'{e} theory}\label{ss:fraisse}

We will briefly review the classical Fra\"{i}ss\'{e} theory before moving on to the generalization that we will need.
Let $\mcL$ be a countable relational language.
For an $\mcL$-structure $\mcA$, we write $A$ to refer to the universe of $\mcA$.
Given two $\mcL$-structures $\mcA$ and $\mcB$, we write $\mcA \preceq_\mcL \mcB$ if $\mcA$ is an $\mcL$-substructure of $\mcB$.
We write $\pi : \mcA \isom \mcB$ if $\pi$ is an isomorphism of $\mcA$ with $\mcB$.
More generally, if $C \subseteq A \cap B$ we write $\pi : \mcA \isom_C \mcB$ if $\pi$ is an isomorphism of $\mcA$ with $\mcB$ that also satisfies $\pi(c) = c$ for all $c \in C$.
Alternatively, we will often say that $\pi$ witnesses $\mcA \isom_C \mcB$, and sometimes we will simply say $\mcA \isom_C \mcB$ when such $\pi$ exists but we have no need to refer to it by name.
A countable $\mcL$-structure $\mcM$ is \emph{ultrahomogeneous} if for any finite $\mcA, \mcB \prec_\mcL \mcM$ and $\pi : \mcA \isom \mcB$, there is an automorphism of $\mcM$ that extends $\pi$.

Let $\mcF$ be a class of finite $\mcL$-structures closed under isomorphism and substructure.
We say that an $\mcL$-structure $\mcM$ is a \emph{limit} of $\mcF$ if the age of $\mcM$ is $\mcF$ and $\mcM$ is ultrahomogeneous.
Such $\mcM$ has an \emph{extension property}: if $\mcA, \mcB \in \mcF$ satisfy $\mcA \prec_\mcL \mcM$ and $\mcA \preceq_\mcL \mcB$, then there is some $\mcB' \isom_A \mcB$ such that $\mcB' \prec_\mcL \mcM$.
Using the extension property and a straightforward back-and-forth argument, one can show that a countable limit (if it exists) is unique up to isomorphism.

We say that $\mcF$ is a Fra\"{i}ss\'{e} class if moreover it satisfies the \emph{amalgamation property}: whenever $\mcA, \mcB, \mcC \in \mcF$ satisfy $\mcC \preceq_\mcL \mcA$ and $\mcC \preceq_\mcL \mcB$ then there is some $\mcA' \isom_C \mcA$ and some $\mcD \in \mcF$ such that $\mcA' \preceq_\mcL \mcD$ and $\mcB \preceq_\mcL \mcD$.
The main result of Fra\"{i}ss\'{e} theory is that such $\mcF$ always has a countable limit.

Say that $\mcF$ satisfies the \emph{disjoint amalgamation property} (also called the strong amalgamation property) iff for every $\mcA, \mcB, \mcC \in \mcF$ with $\mcC \preceq_\mcL \mcA$ and $\mcC \preceq_\mcL \mcB$, there is some $\mcA', \mcD \in \mcF$ such that $\mcA' \isom_C \mcA$ and $\mcA' \preceq_\mcL \mcD$ and $\mcB \preceq_\mcL \mcD$, and moreover $A' \cap B = C$.

\begin{proposition}[Baldwin-Friedman-Koerwien-Laskowski, \cite{BFKL2016}]\label{prop:bfkl}
If $\mcF$ is a Fra\"{i}ss\'{e} class satisfying the disjoint amalgamation property, and $\mcM$ is a limit, then $\Aut(\mcM)$ involves $S_\infty$.
\end{proposition}

The proof proceeds by considering the class $\mcF_c$ of ``colored" versions of structures in $\mcF$.
A colored $\mcF$-structure is a structure $\mcA$ in the extended language $\mcL_c := \mcL \cup \{R_n \mid n \in \omega\}$ where each $R_n$ is a new unary relation such that the $\mcL$-reduct of $\mcA$ is in $\mcF$ and for every $a \in A$ there is exactly one $n$ such that $\mcA \models R_n(a)$.
As a notational abuse, we will write colored $\mcF$-structures as pairs $(\mcA, \oc_\mcA)$ where $\mcA \in \mcF$ and $\oc_\mcA$ is a function $\oc_\mcA : A \rightarrow \omega$, and so given two colored $\mcF$-structures $(\mcA, \oc_\mcA)$ and $(\mcB, \oc_\mcB)$ we have $(\mcA, \oc_\mcA) \preceq_{\mcL_c} (\mcB, \oc_\mcB)$ iff $\mcA \preceq_\mcL \mcB$ and $\oc_\mcB \upharpoonright A = \oc_\mcA$.

The class $\mcF_c$ of colored $\mcF$-structures satisfies downward closure and isomorphism invariance, but need not satisfy the amalgamation property.
For example, consider the case where $\mcF$ is the class of all finite partial matchings (i.e. undirected graphs where each vertex has degree 0 or 1).
This is a Fra\"{i}ss\'{e} class and its limit is the unique perfect matching of a countably-infinite set up to isomorphism.
However, $\mcF_c$ does not satisfy the amalgamation property.
If we had $(\mcC, \oc_\mcC) \preceq_{\mcL_c} (\mcA, \oc_\mcA), (\mcB, \oc_\mcB)$ and there was some $c \in C$ and $a \in A \setminus C$ and $b \in B \setminus C$ with $\mcA$ matching $c$ with $a$ and $\mcB$ matching $c$ with $b$, then any attempt to amalgamate $\mcA$ and $\mcB$ over $\mcC$ would have to identify $a$ with $b$, because $c$ can only be matched to one other vertex.
However if $\oc_\mcA(a) \neq \oc_\mcB(b)$ then this would be impossible.

However, with the additional assumption that $\mcF$ has the disjoint amalgamation property, it is easy to confirm that $\mcF_c$ satisfies the amalgamation property.
Thus we can compute the limit of $\mcF_c$, which would be a pair $(\mcM', \oc)$ for some coloring $\oc : M' \rightarrow \omega$.
Since the limit of a Fra\"{i}ss\'{e} class is unique up to isomorphism, and $\mcM'$ is a limit of $\mcF$, we can assume without loss of generality that $\mcM' = \mcM$.

Next, consider the closed subgroup $H \le \Aut(\mcM)$ of automorphisms of $\mcM$ which permute the colors consistently, i.e. $h \in H$ iff $\oc(a) = \oc(b)$ if and only if $\oc(h \cdot a) = \oc(h \cdot b)$ for every $a, b \in M$, or equivalently, there is some $\sigma_h \in S_\infty$ such that $\oc(h \cdot a) = \sigma_h(\oc(a))$ for every $a \in M$.
There is a natural continuous homomorphism $f : H \rightarrow S_\infty$ sending $h$ to $\sigma_h$, noting that $\sigma_h$ is unique as every color appears somewhere in $M$.
The final step is to show that $f$ is surjective, by fixing an arbitrary $\sigma$ and constructing by a back-and-forth method (utilizing the homogeneity of the coloring $\oc$) an automorphism $h \in H$ such that $\sigma = \sigma_h$.
We'll leave out the details here since we'll be proving a generalization of this argument in the next subsection.

It's natural to ask if the converse to Proposition \ref{prop:bfkl} is true.
Evidently it is not, as one could again consider the case where $\mcF$ is the class of finite partial matchings.
This does not satisfy disjoint amalgamation as every vertex must be matched to at most one other, however its automorphism group involves $S_\infty$ as all of the pairs can be permuted arbitrarily.

Without disjoint amalgamation, we can't assume that the ``colored'' finite partial matchings $\mcF_c$ has the amalgamation property.
However, it's clear what the limit of $\mcF_c$ should be when $\mcF$ is the finite partial matchings.
It should be the countably-infinite perfect matching where every pair of colors (not necessarily distinct) appears infinitely-many times among the matched pairs of vertices.
One could produce this limit using a weakening of Fra\"{i}ss\'{e} theory where the amalgamation property is replaced by the \emph{weak amalgamation property}.
The class of colored finite partial matchings $\mcF_c$ satisfies weak amalgamation, which says that for every $\mcC \in \mcF_c$ there is some $\mcC' \in \mcF_c$ with $\mcC \prec_{\mcL} \mcC'$ such that for every $\mcA, \mcB \in \mcF_c$ with $\mcC' \preceq_\mcL \mcA$ and $\mcC' \preceq_\mcL \mcB$ there is some $\mcA' \isom_{C'} \mcA$ and $\mcD \in \mcF_c$ with $\mcA', \mcB \preceq_\mcL \mcD$.
The idea is if $\mcC \in \mcF_c$ then let $\mcC'$ be a superstructure which gives every vertex a match.

In other words, for this choice of $\mcF$ the amalgamation property works for structures that are definably-closed, or in a slightly different wording, closed with respect to definable closure.
However we come across two problems.
The first is that in the more general situation, the definable closure may be infinite.
This might not be much of a hindrance, since Fra\"{i}ss\'{e} theory was developed in the broader context of ``finitely-generated'' structures.
However the more serious issue is that we will need to restrict the amalgamation property with a closure that is coarser than definable closure.

An example of a coarser closure operator is the notion of pseudo-algebraic closure.
For a countable atomic $\mcL$-structure $\mcM$, the pseudo-algebraic closure of a set $A \subseteq M$ is defined to be the set of all $b \in M$ such that $b \in N$ for every $\mcN \preceq_{\countlogic} \mcM$ with $A \subseteq N$.
It's straightforward to see that if $b$ is in the definable closure of $A$ then it is in the pseudo-algebraic closure of $A$.
This is not quite the correct closure notion we want (we want an even coarser closure operator), but it is reasonably close to the correct one, and indeed our theory will significantly parallel the theory of the pseudo-algebraic closure as developed in \cite{Deissler1997}.

Our solution to this problem is to do two things: re-develop Fra\"{i}ss\'{e} theory with a notion of ``strong substructure'' $\mcA \preceq^* \mcB$ which puts more restrictions than simply $\mcA$ being an $\mcL$-substructure of $\mcB$; and find a suitable closure operator and corresponding weakening of disjoint amalgamation.

We will be doing this in the next subsections, but we describe now for the example of the colored finite partial matchings $\mcF_c$ what we will do differently.
Given $\mcA, \mcB \in \mcF_c$, we say $\mcA \preceq^* \mcB$, i.e. $\mcA$ is a strong substructure of $\mcB$ if $\mcA$ is a substructure of $\mcB$ and moreover for every vertex in $B \setminus A$ that is matched with a vertex in $A$, it must be given color 0.
This will indeed satisfy the right variation of amalgamation property.

\subsection{Strong Fra\"{i}ss\'{e} theory}

\begin{definition}\label{def:strong_substructure}
Let $\mcL$ be a countable relational language.
Let $\mcF$ be a class of finite $\mcL$-structures closed under isomorphism and substructure.
Let $\preceq^*$ be a binary relation on $\mcF$, where we say that $\mcA$ is a \emph{strong substructure} of $\mcB$ if $\mcA \preceq^* \mcB$, which satisfies the following axioms:
\begin{description}
    \item[Preorder]\label{strong:preorder} the strong substructure relation $\preceq^*$ is reflexive and transitive
    \item[Coarsening]\label{strong:coarsening} if $\mcA \preceq^* \mcB$ then $\mcA \preceq_{\mcL} \mcB$;
    \item[Invariance]\label{strong:invariance} if $\mcA \preceq^* \mcB$ and $\pi : \mcB \isom \mcC$ is an isomorphism, then $\pi[\mcA] \preceq^* \mcC$;
    \item[Downward closure]\label{strong:downward_closure} if $\mcA \preceq^* \mcB$ and $\mcA \preceq_\mcL \mcC \preceq_\mcL \mcB$ then $\mcA \preceq^* \mcC$;
    \item[Positivity]\label{strong:positivity} $\emptyset \preceq^* \mcA$ for every $\mcA \in \mcF$; and
    \item[No maximal models]\label{strong:nmm} there are no maximal models with respect to $\preceq^*$, i.e. for every $\mcA \in \mcF$ there is some $\mcB \in \mcF$ with $\mcA \prec^* \mcB$.
\end{description}
\end{definition}

One should check that for our earlier example of the colored finite partial-matchings, the example of strong substructure we described will satisfy all the above axioms.
Furthermore, with this example in mind we can see that we should not insist that if $\mcA \preceq^* \mcB$ and $\mcC \prec_\mcL \mcA$ then $\mcC \preceq^* \mcB$.

We write $\mcB \isom \mcC$ iff there is an isomorphism $\pi : B \rightarrow C$ between $\mcB$ and $\mcC$.
More generally, if $A \subseteq B \cap C$, then we write $\mcB \isom_A \mcC$ iff there is an isomorphism $\pi : B \rightarrow C$ between $\mcB$ and $\mcC$ satisfying $\pi(a) = a$ for every $a \in A$.

Let $\mcM$ be an infinite $\mcL$-structure (which would thus not be in $\mcF$).
Given $\mcA \in \mcF$ with $\mcA \prec_{\mcL} \mcM$ we say $\mcA$ is a \emph{strong substructure} of $\mcM$, denoted $\mcA \prec^* \mcM$ if for any $\mcB \in \mcF$ with $\mcA \preceq_\mcL \mcB \prec_\mcL \mcM$ we have $\mcA \preceq^* \mcB$.

\begin{lemma}\label{lem:strong_substructure}
If $\mcM$ is an $\mcL$-structure and $\mcA, \mcB \in \mcF$ satisfy $\mcA \preceq^* \mcB$ and $\mcB \prec^* \mcM$ then $\mcA \prec^* \mcM$.
Moreover, if $\mcA \prec^* \mcM$ and $\sigma$ is an automorphism of $\mcM$ then $\sigma[\mcA] \prec^* \mcM$.
\end{lemma}

\begin{proof}
For the first claim, fix arbitrary $\mcC \in \mcF$ with $\mcA \preceq_\mcL \mcC \prec_\mcL \mcM$.
Take $\mcD := \mcM \upharpoonright (B \cup C)$.
Then $\mcB \preceq_\mcL \mcD$ and since $\mcB \prec^* \mcM$ we have $\mcB \preceq^* \mcD$.
Thus by transitivity of $\prec^*$ we have $\mcA \preceq^* \mcD$.
Since we have $\mcA \preceq_\mcL \mcC \preceq_\mcL \mcD$ then by downward closure we have $\mcA \preceq^* \mcC$ as desired.

For the second claim, fix arbitrary $\mcA \prec^* \mcM$ and $\sigma$ an automorphism of $\mcM$.
To see that $\sigma[\mcA] \prec^* \mcM$ fix any $\mcB \in \mcF$ satisfying $\sigma[\mcA] \preceq_\mcL \mcB \prec_\mcL \mcM$.
Then we have $\mcA \preceq_\mcL \sigma^{-1}[\mcB] \prec_\mcL \mcM$.
Since $\mcA \prec^* \mcM$ we have $\mcA \preceq^* \sigma^{-1}[\mcB]$ and thus by invariance we have $\sigma[\mcA] \preceq^* \mcB$ as desired.
\end{proof}

We say that $\mcM$ is \emph{ultrahomogeneous on strong substructures} if whenever $\mcA, \mcA' \in \mcF$ satisfy $\mcA \prec^* \mcM$ and $\mcA' \prec^* \mcM$ then any isomorphism $\pi : \mcA \isom \mcA'$ can be extended to an automorphism of $\mcM$.

\begin{definition}\label{def:strong_limit}
Given a countable relational language $\mcL$ and $\mcF$ a class of finite $\mcL$-structures closed under isomorphism and substructure, and $\preceq^*$ a notion of strong substructure, we say that a countable $\mcL$-structure $\mcM$ is a \emph{strong limit} of $(\mcF, \preceq^*)$ if
\begin{enumerate}
\item[(i)] For every finite $C_0 \subseteq M$ there is some $\mcA \in \mcF$ with $\mcA \prec^* \mcM$ and $C_0 \subseteq A$;
\item[(ii)] For every $\mcA \in \mcF$ there is some $\mcA' \isom \mcA$ such that $\mcA' \prec^* \mcM$; and
\item[(iii)] $\mcM$ is ultrahomogeneous on strong substructures.
\end{enumerate}
\end{definition}

\begin{lemma}\label{lem:extension_property}
If $\mcM$ is the strong limit of $(\mcF, \preceq^*)$ then it has the following \emph{extension property}: if $\mcA, \mcB \in \mcF$ satisfy $\mcA \prec^* \mcM$ and $\mcA \preceq^* \mcB$ then there is some $\mcB' \isom_A \mcB$ such that $\mcA \preceq^* \mcB' \prec^* \mcM$. 
\end{lemma}

\begin{proof}
By (ii) there is some $\pi : \mcB' \isom \mcB$ with $\mcB' \prec^* \mcM$.
Letting $\mcA' := \pi^{-1}[\mcA]$ we have by invariance that $\mcA' \prec^* \mcB'$ and thus by Lemma \ref{lem:strong_substructure} that $\mcA' \prec^* \mcM$.
Thus by (iii) there is an automorphism $\sigma$ of $\mcM$ extending $\pi \upharpoonright A'$.
Let $\mcB'' := \sigma[\mcB']$.
Then by Lemma \ref{lem:strong_substructure} we have $\mcB'' \prec^* \mcM$ and by invariance we have $\mcA = \sigma[\mcA'] \prec^* \sigma[\mcB'] = \mcB''$.
Moreover we have $\pi \circ \sigma^{-1} : \mcB'' \isom_{A} \mcB$ as desired.
\end{proof}

\begin{proposition}\label{prop:unique_limit}
If it exists, the strong limit of $(\mcF, \preceq^*)$ is unique up to isomorphism.
\end{proposition}

\begin{proof}
Let $\mcM$ and $\mcN$ be two strong limits of $(\mcF, \preceq^*)$.
Since they are countable, it suffices to show that the set
\[ \mcI = \{\pi : \mcA \isom \mcB \mid \mcA, \mcB \in \mcF, \; \mcA \prec^* \mcM, \; \mcB \prec^* \mcN\}\]
has the back-and-forth property.

In particular, the first thing we need to show is that if $\pi : \mcA \isom \mcB$ is in $\mcI$ and $c \in M$ then there is an extension $\pi' : \mcA' \isom \mcB'$ of $\pi$ where $\mcA \prec_\mcL \mcA' \prec^* \mcM$ and $\mcB \prec_\mcL \mcB' \prec^* \mcN$ and $c \in A'$.
Indeed, apply Definition \ref{def:strong_limit}.(i) to find some $\mcA' \in \mcF$ with $A \cup \{c\} \subseteq A'$ where $\mcA' \prec^* \mcM$.
Extend $\pi$ arbitrarily to some $\pi' : \mcA' \isom \mcB'$ where $\mcB \prec_\mcL \mcB'$.
Since $\mcA \prec^* \mcM$ we have by definition that $\mcA \prec^* \mcA'$ and thus by invariance we have $\mcB \prec^* \mcB'$.
Thus by the extension property Lemma \ref{lem:extension_property} there is some $\sigma : \mcB'' \isom_B \mcB'$ such that $\mcB'' \prec^* \mcN$.
Then $\sigma^{-1} \circ \pi' : \mcA' \isom \mcB''$ is the desired extension of $\pi$.

The second thing we need to show is that if $\pi : \mcA \isom \mcB$ is in $\mcI$ and $d \in N$ then there is an extension $\pi' : \mcA' \isom \mcB'$ of $\pi$ where $\mcA \prec_\mcL \mcA' \prec^* \mcM$ and $\mcB \prec_\mcL \mcB' \prec^* \mcN$ and $d \in B'$.
This follows by the same argument.

Also of course it is nonempty since it contains the empty function by positivity.
\end{proof}

We say that $\mcF$ satisfies the \emph{amalgamation property on strong substructures} if for any $\mcA, \mcB, \mcC \in \mcF$, if $\mcA \preceq^* \mcB$ and $\mcA \preceq^* \mcC$ then there is some $\mcC' \in \mcF$ satisfying $\mcC' \isom_A \mcC$ and $\mcD \in \mcF$ such that $\mcB \preceq^* \mcD$ and $\mcC' \preceq^* \mcD$.

\begin{proposition}\label{prop:exists_limit}
If $(\mcF, \preceq^*)$ satisfies the amalgamation property on strong substructures then it has a strong limit.
\end{proposition}

\begin{proof}
Let $(\mcA_n, \mcB_n)_{n \in \omega}$ be a sequence of pairs $\mcA_n, \mcB_n$ from $\mcF$ with $\mcA_n \prec^* \mcB_n$ and $A_n \subseteq \omega$ such that for any pair $\mcA, \mcB \in \mcF$ with $\mcA \prec^* \mcB$ and $A \subseteq \omega$ there are infinitely-many $n$ such that $\mcA_n = \mcA$ and $\mcB_n \isom_{\mcA_n} \mcB$.

We inductively construct a sequence $(\mcC_n)_{n \in \omega}$ from $\mcF$ satisfying for each $n$
\begin{enumerate}
\item[(a)] $C_n \subseteq \omega$
\item[(b)] $\mcC_n \preceq^* \mcC_{n+1}$
\item[(c)] $n \in C_{2n+1}$
\item[(d)] if $\mcA_n \preceq^* \mcC_{2n+1}$ then there is $\mcB' \isom_{\mcA_n} \mcB_n$ such that $\mcB' \preceq^* \mcC_{2n+2}$.
\end{enumerate}
Letting $\mcC_0 = \emptyset$, suppose $(\mcC_i)_{i \le 2n}$ has been defined as above.
By no maximal models from Definition \ref{def:strong_substructure}, find some $\mcC_{2n+1} \in \mcF$ with $\mcC_{2n} \prec^* \mcC_{2n+1}$. Since $\mcF$ is closed under isomorphism, we can assume without loss of generality that $C_{2n+1} \subseteq \omega$ and $n \in C_{2n+1}$.
If $\mcA_n \preceq^* \mcC_{2n+1}$ then by the amalgamation property on strong substructures, find $\mcC_{2n+2} \in \mcF$ and $\mcB_n' \in \mcF$ such that $\mcB_n' \isom_{\mcA_n} \mcB_n$ and $\mcC_{2n+1} \preceq^* \mcC_{2n+2}$ and $\mcB_n' \preceq^* \mcC_{2n+2}$.
Since $\mcF$ is closed under isomorphism, without loss of generality we can assume that $C_{2n+2} \subseteq \omega$.
Thus we have established such a sequence $(\mcC_n)_{n \in \omega}$ exists.

Let $\mcM$ be the limit of the chain $\mcC_0 \preceq_\mcL \mcC_1 \preceq_\mcL \dots$.
By (a) and (c) the universe of $\mcM$ is $\omega$.
We also observe that each $\mcC_n \prec^* \mcM$.
Indeed, fixing arbitrary finite $\mcD$ with $\mcC_n \preceq_\mcL \mcD \prec_\mcL \mcM$ we can choose $m \ge n$ large enough such that $D \subseteq C_m$ (which exists by (a)) and so $\mcC_n \prec^* \mcC_m$ by transitivity and (b), and so by downward closure we have $\mcC_n \prec^* \mcD$ as desired.

We must show that $\mcM$ is ultrahomogeneous on strong substructures.
In particular, as in the proof of Proposition \ref{prop:unique_limit}, we just need to see that the set
\[ \mcI = \{\pi : \mcA \isom \mcB \mid \mcA, \mcB \in \mcF, \; \mcA, \mcB \prec^* \mcM\}\]
has the back-and-forth property.

Given $\pi : \mcA \isom \mcB \in \mcI$ and $c \in M$ we need to find an extension $\pi' : \mcA' \isom \mcB'$ of $\pi$ with $c \in A'$.
By (c) choose $n$ large enough such that $A \cup\{c\} \subseteq C_n$.
Since $\mcA \prec^* \mcM$ by definition we have $\mcA \prec^* \mcC_n$.
Extend $\pi$ arbitrarily to $\hat{\pi} : \mcC_n \isom \hat{\mcB}$ where $\hat{\mcB} \in \mcF$ satisfies $\mcB \preceq_\mcL \hat{\mcB}$.
By invariance we have $\mcB \preceq^* \hat{\mcB}$.
Choose $m$ such that $\mcA_m = \mcB$ and $\mcB_m \isom_{\mcA_m} \hat{\mcB}$, and large enough such that also $B \subseteq C_m$.
Then by (d) there is $\mcD \isom_{\mcB} \hat{\mcB}$ such that $\mcD \preceq^* \mcC_{2m+2}$.
Let $\sigma : \hat{\mcB} \isom_{\mcB} \mcD$ and define $\pi^* :\mcC_{n} \isom \mcD$ by $\pi^* = \sigma \circ \hat{\pi}$.
We have that $\pi^* \upharpoonright A = \pi$ and moreover $\mcD \prec^* \mcM$ by Lemma \ref{lem:strong_substructure} since $\mcD \preceq^* \mcC_{2m+2}$.
Thus $\pi^*$ is the desired extension of $\pi$.

The other direction of the back-and-forth property is similar.
It is nonempty since by positivity we have $\emptyset \prec^* \mcM$ and so it has the empty function.
\end{proof}

\subsection{Closure operators and independence relations}

We begin by considering definable closure and pseudo-algebraic closure in an abstract sense, as \emph{closure operators}.

A \emph{closure operator} on a set $I$ is a function $\cl: \mathcal{P}(I) \rightarrow \mathcal{P}(I)$ satisfying for every $A, B \subseteq I$:
\begin{enumerate}
    \item $A \subseteq \cl(A)$;
    \item $\cl(A) \subseteq \cl(B)$ whenever $A \subseteq B$; and
    \item $\cl(\cl(A)) = \cl(A)$.
\end{enumerate}
Note that sometimes the axioms of closure operators demand that $\cl(\emptyset) = \emptyset$, though we will not (and should not) demand that here.
In the case that $\cl(A) = A$, we say that $A$ is \emph{closed}.
We say that  $\cl$ is \emph{non-trivial} when $\cl(\emptyset) \neq I$.
The closure operators considered in this paper will have \emph{finite character}, meaning that $\cl(A)$ is the union of $\cl(A_0)$ where $A_0$ ranges over finite subsets of $A$. 
We will adopt the standard notation to sometimes write $A \cup B$ as simply $AB$  and $A \cup \{b\}$ as simply $Ab$, for $A, B \subseteq I$ and $b \in I$.

Given $B, C$ finite subsets of $I$, say that $b \in B \setminus \cl(C)$ is \emph{minimal in $B$ over $C$} iff for every $b' \in (B \cap \cl(bC)) \setminus \cl(C)$, we have $b \in \cl(b'C)$.

\begin{lemma}
Suppose $\cl$ is a closure operator on $I$ and $B \supseteq C$ are two finite subsets of $I$ such that $B \not\subseteq \cl(C)$. Then there exists some $b \in B \setminus \cl(C)$ which is minimal in $B$ over $C$.
\end{lemma}

\begin{proof}
Define on $B \setminus \cl(C)$ a binary relation $\le$ by $b \le b'$ iff $b \in \cl(b'C)$.
This is easily seen to be a preorder.
As $B \setminus \cl(C)$ is finite, we may choose a $\le$-minimal element.
\end{proof}

Given such a closure operator $\cl$ on a set $I$ and subsets $A, B, C \subseteq I$ , write 
\[A \forkindep[C] B \quad \text{iff} \quad A \cap \cl(B) \subseteq \cl(C) \; \text{and} \; B \cap \cl(A) \subseteq \cl(C).\]
We call $\forkindep$ \textbf{the independence relation derived from $\cl$}.

\begin{lemma}\label{lem:fork_props}
Suppose $\forkindep$ is the independence relation derived from a closure operator $\cl$ on a set $I$. For any $A, B, C \subseteq I$, the following hold:
\begin{description}
    \item[Existence] If $B \subseteq \cl(C)$ then $A \forkindep[C] B$.
    \item[Symmetry] We have
    \[ A \forkindep[C] B \quad \text{iff} \quad B \forkindep[C] A\]
    \item[Monotonicity] Whenever $A' \subseteq A$, $B' \subseteq B$, and $C' \supseteq C$, we have
    \[ A \forkindep[C] B\quad  \text{implies} \quad A' \forkindep[C'] B' \]
    \item[Finite character] $A \forkindep[C] B$ iff for every finite $A_0 \subseteq A$ and $B_0 \subseteq B$ there exists finite $C_0 \subseteq C$ such that $A_0 \forkindep[C_0] B_0$.
    \item[Weak transitivity] For finite $A, B, C$ with $C \subseteq A$ and $C \subseteq B$ if $b \in B \setminus \cl(C)$ is minimal in $B$ over $C$, then 
    \[ A \forkindep[C] bC \quad \text{and} \quad  A \forkindep[bC] B \quad \text{together imply} \quad A \forkindep[C] B. \]
\end{description}
\end{lemma}

\begin{proof}
Let $\cl$ be the closure operator from which $\forkindep$ is derived.
Existence, symmetry, and monotonicity are immediate from the definition of $\forkindep$ and properties of $\cl$.

To prove finite character, let's first assume $A \forkindep[C] B$ and let $A_0 \subseteq A$ and $B_0 \subseteq B$ be finite.
For each $a \in A_0 \cap \cl(B)$ by assumption we have $a \in \cl(C)$ and thus by finite character of $\cl$ there is some finite $C^0_a \subseteq C$ such that $a \in \cl(C^0_a)$.
Similarly fix for each $b \in B_0 \cap \cl(A)$ a finite set $C^1_b \subseteq C$ such that $b \in \cl(C^1_b)$.
Then letting $C_0 = [\bigcup_{a} C^0_a] \cup [\bigcup_{b} C^1_b]$ we see that $C_0 \subseteq C$ is finite and $A_0 \forkindep[C_0] B_0$.
On the other hand, if $A \not\forkindep[C] B$ then either there is some $a \in A \cap \cl(B)$ with $a \not\in \cl(C)$ or some $b \in B \cap \cl(A)$ with $b \not\in \cl(C)$.
In the first case, by finite character of $\cl$ we can find some finite $B_0 \subseteq B$ such that $a \in A \cap \cl(B_0)$ and thus $a \not\forkindep[C_0] B_0$ for any finite $C_0 \subseteq C$ by monotonicity.
The other case is similar.

Now we prove weak transitivity.
Let $A, B, C \subseteq I$ be finite, $b \in B \setminus \cl(C)$ be minimal in $B$ over $C$, and $A \forkindep[C] bC$ and $A \forkindep[bC] B$.
To show that $A \forkindep[C] B$ we first fix an arbitrary $a \in A \cap \cl(B)$. By $A \forkindep[bC] B$ we conclude $a \in \cl(bC)$ and then by $A \forkindep[C] bC$ we conclude $a \in \cl(C)$ as desired. 
Second, fix an arbitrary $b' \in B \cap \cl(A)$ and assume for contradiction $b' \not\in \cl(C)$. 
Then by $A \forkindep[bC] B$ we have that $b' \in \cl(bC)$. 
By minimality of $b$, we have $b \in \cl(b'C)$.
Since $b' \in \cl(A)$ we have $b \in \cl(A)$ and so by $A \forkindep[C] bC$, we have $b \in \cl(C)$.
This in turn implies $b' \in \cl(C)$, which is a contradiction.
\end{proof}

\subsection{Invariant closure operators}

We will ultimately be defining closure operators on structures, and we will need them to have invariance properties with respect to automorphisms.
In the meantime we will continue our discussion of closure operators in the abstract, but now taking into account the existence of a group action on the underlying set.

When there is some group action $P \curvearrowright I$, we say that a closure operator $\cl$ on $I$ is 
\emph{$P$-invariant} (or just \emph{invariant} when there is no chance of confusion) iff for every $A \subseteq I$ and $\pi \in P$, we have $\pi[\cl(A)] = \cl(\pi[A])$. 
Given sets $A, B, C \subseteq I$, we write $A \isom^P_C B$ to denote that there is some $\pi \in P$ such that $\pi[A] = B$ and $\pi(c) = c$ for every $c \in C$.
In practice, we will drop the superscript.

The following property of invariant closure operators comes in handy.
\begin{lemma}\label{lem:one_side}
Suppose $P \curvearrowright I$ and $\cl$ is an invariant closure operator on $I$. Then for any $a, b \in I$ and $C \subseteq I$ finite, if for every $a' \isom_{C} a$ we have $a' \in \cl(bC)$, then for every $a' \isom_{C} a$ and $b' \isom_{C} b$ we have $a' \in \cl(b' C)$.
\end{lemma}

\begin{proof}
Let $a' \isom_C a$ and $\sigma : b' \isom_C b$. By assumption $\sigma(a') \in \cl(bC)$ and by invariance of $\cl$ we have $a' \in \cl(b'C)$.
\end{proof}

We will usually be interested in the case that we have an $\mcL$-structure $\mcM$ where $\mcL$ is a countable relational language, and $\cl$ is a closure operator on $M$ which is invariant with respect to the natural action $\Aut(\mcM) \curvearrowright M$.
Definable closure and pseudo-algebraic closure on a structure $\mcM$ are easily seen to be invariant closure operators with respect to this action.

\subsection{The disjointifying property}

We now define a new property of invariant closure operators on sets $I$ with group action $P \curvearrowright I$.
Let $\forkindep$ be the independence relation derived from an invariant closure operator $\cl$ on $I$.
We say that $\cl$ is \textbf{disjointifying} iff it is invariant, finite character, and whenever $A, B, C$ are finite subsets of $I$, with $C \subseteq A$ and $C \subseteq B$ there exists a finite subset $A'$ of $I$ such that $A' \isom_C A$ and $A' \forkindep[C] B$.

In the following proposition, we prove three additional statements equivalent to an invariant closure operator $\cl$ being disjointifying. 
Statement (3) below shows that in the definition of the disjointifying property, it is enough to consider $A$ and $B$ when $|A \setminus C| = |B \setminus C| = 1$. 
Statement (4) is helpful for verifying the disjointifying property as it allows one to check two easier statements separately.

\begin{proposition}\label{prop:cl_equivalence}
Suppose $\cl$ is an invariant finite-character closure operator on $P \curvearrowright I$ with derived independence relation $\forkindep$. The following are equivalent:
\begin{enumerate}
    \item $\cl$ is disjointifying;
    \item for any finite $C \subseteq I$ and $B \subseteq I$ with $C \subseteq B$ and $a \in I$, there is some $a' \isom_C a$ such that $a'C \forkindep[C] B$; and
    \item for any finite $C \subseteq I$ and $a, b \in I$, there is some $a' \isom_C a$ such that $a'C \forkindep[C] bC$;
    \item for any finite $C \subseteq I$ and $a, b \in I$ with $a \not\in \cl(C)$, the following both hold:
    \begin{enumerate}
        \item there is some $a' \isom_C a$ such that $a'C \forkindep[C] aC$; and
        \item there is some $a' \isom_C a$ such that $a' \not\in \cl(bC)$.
    \end{enumerate}
\end{enumerate}

\end{proposition}

\begin{proof}
Clearly (1) implies (2), (2) implies (3), and (3) implies (4).
However showing (4) implies (1) is not easy, and instead we show each step backwards. In order we show (3) implies (2), (2) implies (1), then (4) implies (3).

We prove (3) implies (2).
We will induct on $|B \setminus \cl(C)|$.
For the base case, let $B, C \subseteq I$ finite with $C \subseteq B$ and $B \subseteq \cl(C)$.
Then for any $a \in I$ by existence we have $aC \forkindep[C] B$.

For the induction step, let $B, C \subseteq I$ be finite and assume $|B \setminus \cl(C)| > 0$ and that the claim holds for lesser values of $|B \setminus \cl(C)|$.
Let $a \in I$ and fix some $b_0 \in B \setminus \cl(C)$ minimal in $B$ over $C$.
By assumption find $a' \isom_C a$ such that $a'C \forkindep[C] b_0C$. 
By the induction hypothesis, find some $a'' \isom_{b_0C} a'$ as witnessed by some $\pi \in P$ such that $a''C \forkindep[b_0C] B$.
Clearly $a'' \isom_C a$.
By invariance with respect to $\pi$ we have $a''C \forkindep[C] b_0C$, and then by weak transitivity we have $a''C \forkindep[C] B$ as desired.

The proof that (2) implies (1) is similar.
We induct on $|A \setminus \cl(C)|$. 
If this is 0, then as before the conclusion follows by existence.
Otherwise, assume $|A \setminus \cl(C)| > 0$ and the claim holds for lesser values.
Let $a_0 \in A \setminus \cl(C)$ minimal in $A$ over $C$.
Find $a_0' \isom_C a_0$ such that $a_0' \forkindep[C] B$, as witnessed by some $\pi \in P$.
Define $A' := \pi[A]$ so that $A' \isom_C A$.
By invariance of $\cl$, we have $a_0' \in A'$ is minimal in $A'$ over $C$.
By the induction hypothesis, find $A'' \isom_{a_0'C} A'$ such that $A'' \forkindep[a_0'C] B$.
Observe that $a_0'$ is minimal in $A''$ over $C$ as well, by invariance of $\cl$.
By weak transitivity we have $A'' \forkindep[C] B$ as desired.

Finally we see that (4) implies (3). 
We prove this in a few steps via the following intermediate statements
\begin{enumerate}
    \item[(3*)] Given $a, C$ suppose there exists some $b \not\in \cl(C)$ such that for every $b' \isom_C b$, either $a \in \cl(b'C)$ or $b' \in \cl(aC)$. Then $a \in \cl(C)$.
    \item[(4*)] Given $a, C$ suppose either of the following holds:
    \begin{enumerate}
        \item[(a)] there exists some $b$ such that for every $b' \isom_C b$, $a \in \cl(b'C)$; or
        \item[(b)] for any $a' \isom_C a$ either $a \in \cl(a'C)$ or $a' \in \cl(aC)$.
    \end{enumerate}
    Then $a \in \cl(C)$.
\end{enumerate}
We show that (4) implies (4*) and (4*) implies (3*) and then (3*) implies (3).

Let's show that (4) implies (4*).
Let $a \in I$ and $C \subseteq I$ finite.
First let's suppose (4*).(a) holds, i.e. we have some $b \in I$ such that for every $b' \isom_C b$ we have $a \in \cl(b'C)$.
If we assume for the sake of contradiction that $a \not\in \cl(C)$ then by (4).(b) there is some $a' \isom_C a$ as witnessed by some $\pi \in P$ such that $a' \not\in \cl(bC)$.
Letting $b' := \pi^{-1}(b)$ we have by invariance of $\cl$ that $a \not\in \cl(b'C)$ where $b' \isom_C b$, which contradicts our choice of $b$.
Next, let's suppose (4*).(b) holds, i.e. for any $a' \isom_C a$ either $a \in \cl(a'C)$ or $a' \in \cl(aC)$.
If we assume for the sake of contradiction that $a \not\in \cl(C)$ then by (4).(a) there is some $a' \isom_C a$ such that $a'C \forkindep[C] aC$, or in other words
\[ a'C \cap \cl(aC) \subseteq \cl(C) \quad \text{and} \quad aC \cap \cl(a'C) \subseteq \cl(C). \]
However if $a \in \cl(a'C)$ then by the right inclusion we have $a \in \cl(C)$ which is a contradiction, and if $a' \in \cl(aC)$ then by the left inclusion we have $a' \in \cl(C)$ and thus $a \in \cl(C)$ by invariance of $\cl$, which is also a contradiction.

Next we show that (4*) implies (3*)
Suppose we're given $a, C$ and some $b \not\in \cl(C)$ such that for every $b' \isom_C b$, either $a \in \cl(b'C)$ or $b' \in \cl(aC)$. 
Assume for contradiction that $a \not\in \cl(C)$. 

We in fact have that for every $a' \isom_C a$ and $b' \isom_C b$, either $a' \in \cl(b'C)$ or $b' \in \cl(a'C)$.
Indeed, let $\pi$ witness $a' \isom_C a$ and let $b'' := \pi(b')$.
Then we have $b'' \isom_C b$ and so either $a \in \cl(b''C)$ or $b'' \in \cl(aC)$.
In the first case by invariance of $\cl$ with respect to $\pi^{-1}$ we have $a' \in \cl(b'C)$ and in the second case for the same reason we have $b' \in \cl(a'C)$ as desired.

Define a relation $\le$ on the set of all $a'$ with $a' \isom_C a$ by
\[ a' \le a'' \quad \text{iff} \quad \forall b' \isom_C b, \; (a'' \in \cl(b'C) \rightarrow a' \in \cl(b'C)). \]
We argue that this is in fact a linear preorder.
Reflexivity and transitivity of $\le$ are clear. 
To verify trichotomy, suppose for contradiction there exist some $a_0, a_1 \isom_C a$ with $a_0 \not\le a_1$ and $a_1 \not\le a_0$ or in other words, there are some $b_0, b_1 \isom_C b$ with 
\[\text{[i]} \; a_1 \in \cl(b_0C) \quad \text{and} \quad \text{[ii]} \; a_0 \not\in \cl(b_0C)\]
and 
\[\text{[iii]} \; a_1 \not\in \cl(b_1C) \quad \text{and} \quad \text{[iv]} \; a_0 \in \cl(b_1C).\]
By [iii], we have $b_1 \in \cl(a_1C)$ and thus by [iv] we have $a_0 \in \cl(a_1C)$. 
But that implies $a_0 \in \cl(b_0C)$ by [i], which contradicts [ii].

Now we show that $a' \le a''$ implies that $a' \in \cl(a''C)$.
Since either $a \le a'$ or $a' \le a$ for every $a' \isom_C a$, we would be done because (4*).(b) would hold and we would have $a \in \cl(C)$, a contradiction.

Assume $a' \le a''$.
We may find some $b' \isom_C b$ such that $a'' \in \cl(b'C)$.
Indeed, otherwise we would have $b' \in \cl(a''C)$ for every $b' \isom_C b$.
Thus $b' \in \cl(a'''C)$ for every $a''' \isom_C a$ by Lemma \ref{lem:one_side}, allowing us to conclude $b \in \cl(C)$ by (4*).(a), a contradiction.

Then of course by invariance we have $a'' \in \cl(b''C)$ for every $b'' \isom_{a''C} b'$ and since $a' \le a''$ we thus have $a' \in \cl(b''C)$ for every $b'' \isom_{a''C} b'$.
By (4*).(a), we can conclude $a' \in \cl(a''C)$ as desired.

Finally, we show that (3*) implies (3).
Assume (3*) and let finite $C \subseteq I$ and $a \in I$ be arbitrary.
Assume for the sake of contradiction that there is some $b \in I$ such that for every $a' \isom_C a$, we have $a'C \not\forkindep[C] bC$.
We must have $a, b \not\in \cl(C)$ otherwise $aC \forkindep[C] bC$ would follow by existence and symmetry.
We claim that
\[ \forall b' \isom_C b, \; a \in \cl(b'C) \quad \text{or} \quad b' \in \cl(aC). \]
This suffices as by (3*) we would conclude $a \in \cl(C)$ which would be a contradiction.
To that end, fix arbitrary $b' \isom_C b$ as witnessed by some $\pi \in P$.
Define $a' := \pi(a)$ in which case we have $a' \isom_C a$ and so $a'C \not\forkindep[C] bC$ by assumption.
By invariance under $\pi$ we have $aC \not\forkindep[C] b'C$ and thus we have
\[ aC \cap \cl(b'C) \not\subseteq \cl(C) \quad \text{or} \quad b'C \cap \cl(aC) \not\subseteq \cl(C). \]
In the first case, we conclude that $a \in \cl(b'C)$ and in the second case we conclude $b' \in \cl(aC)$ as desired.
\end{proof}

Recall that $\cl$ is called non-trivial iff $\cl(\emptyset) \neq I$.
In fact, if $\cl$ is non-trivial and disjointifying then $I \neq \cl(B_0)$ for any finite $B_0 \subseteq I$.
Indeed, if we fix $a \not\in \cl(\emptyset)$ then by the disjointifying property we find $a' \isom_\emptyset a$ such that $a' \forkindep[\emptyset] B_0$, we would have to conclude that $a' \not\in B_0$ since by invariance we would also have $a' \not\in \cl(\emptyset)$.

\subsection{Disjointifying closure operators and involving $S_\infty$}

We now finish this section by proving the following theorem:

\begin{theorem}\label{thm:disjointifying_closure_implies_involve}
Let $I$ be a countably-infinite set and $P \le S_I$ closed with the natural action $P \curvearrowright I$. If $I$ has a nontrivial disjointifying closure operator $\cl$, then $P$ involves $S_\infty$.
\end{theorem}

The remainder of this section will be dedicated to proving this.

\begin{proof}
Let $\mcL$ be the countable relational language with an $n$-ary relation $R_{\bar{a}}$ for each $n \in \omega$ and $\bar{a} \in I^n$.
We define an $\mcL$-structure $\mcM$ living on $M := I$ where for every $\bar{a}, \bar{b} \in I^{<\omega}$, we have that $R_{\bar{a}}(\bar{b})^{\mcM}$ holds iff there exists some $g \in P$ with $g \cdot \bar{a} = \bar{b}$.
Then $P$ is isomorphic to $\Aut(\mcM)$.
It's easy to check that $\mcM$ is ultrahomogeneous, thus if we let $\mcF$ be the age of $\mcM$, we have that $\mcF$ is a Fra\"{i}ss\'{e} class with limit $\mcM$.

For each $\mcA \in \mcF$, we have the disjointifying closure operator $\cl_{\mcA}$ on $A$ that $\mcA$ inherits as a substructure of $\mcM$.
In particular, given $A_0 \subseteq A$ and $a \in A$ we say $a \in \cl_{\mcA}(A_0)$ iff for some $\pi : \mcA \isom \mcB \prec_\mcL \mcM$ we have $\pi(a) \in \cl(\pi[A_0])$.
This is well-defined: indeed, suppose we also have $\sigma : \mcA \isom \mcC \prec_\mcL \mcM$.
Define $B_0 := \pi[A_0]$ and $b = \pi(a)$ and $C_0 := \sigma[A_0]$ and $c = \sigma(a)$.
Then we have $(\sigma \circ \pi^{-1}) \upharpoonright B_0 : \mcB_0 \isom \mcC_0$, and so by ultrahomogeneity of $\mcM$ we can extend this to some $\rho \in \Aut(\mcM)$.
Let $b' := \rho(b)$.
It suffices to show that $b' \isom_{C_0} c$ because by invariance of $\cl$ we would have $c \in \cl(C_0)$ iff $b' \in \cl(C_0)$ iff $b = \rho^{-1}(b') \in \cl(\rho^{-1}[C_0]) = \cl(B_0)$.
Indeed, define $\theta : b'C_0 \rightarrow cC_0$ by $\theta := (\sigma \circ \pi^{-1} \circ \rho^{-1}) \upharpoonright b'C_0$.
This extends to an automorphism of $\mcM$ by ultrahomogeneity, and $\theta(b') = c$ and $\theta(c_0) = c_0$ for all $c_0 \in C_0$ since $\rho^{-1}(c_0) = \sigma(\pi^{-1}(c_0))$.

\begin{claim}\label{claim:little_closure}
The following are true for any $\mcA, \mcB \in \mcF$:
\begin{enumerate}
\item if $\pi : \mcA \isom \mcB$ then for any $A_0 \subseteq A$ we have $\cl_{\mcB}(\pi[A_0]) = \pi[\cl_{\mcA}(A_0)]$;
\item if $\mcA \preceq_\mcL \mcB$ then for any $A_0 \subseteq A$ we have $\cl_\mcA(A_0) = \cl_\mcB(A_0) \cap A$;
\end{enumerate}
\end{claim}

\begin{proof}
For (1), let $\sigma : \mcA \isom \mcA'$ with $\mcA' \prec_\mcL \mcM$.
Then $\sigma \circ \pi^{-1} : \mcB \isom \mcA'$.
Then $b \in \cl_\mcB(\pi[A_0])$ is equivalent by definition to $\sigma \circ \pi^{-1}(b) \in \cl(\sigma[A_0])$ which is equivalent by definition to $\pi^{-1}(b) \in \cl_\mcA(A_0)$, or in other words, $b \in \pi[\cl_\mcA(A_0)]$.

For (2), let $\sigma : \mcB \isom \mcB'$ where $\mcB' \prec_\mcL \mcM$.
Then $\sigma \upharpoonright A : \mcA \isom \mcA'$ where $\mcA' := \sigma[\mcA] \prec_\mcL \mcM$.
Then for any $a \in A$, by definition $a \in \cl_\mcA(A_0)$ if and only if $\sigma(a) \in \cl(\sigma[A_0])$ if and only if $a \in \cl_\mcB(A_0)$.
\end{proof}

Let $\mcF_c$ be the class of colored $\mcF$-structures, which as before we represent as pairs $(\mcA, \oc_\mcA)$ where $\mcA \in \mcF$ and $\oc_\mcA : A \rightarrow \omega$.
We will think of 0 as a special color, the ``null'' color.
We define a notion of strong substructure as follows.
We say $(\mcA, \oc_\mcA)$ is a strong substructure of $(\mcB, \oc_\mcB)$, denoted $(\mcA, \oc_\mcA) \preceq^* (\mcB, \oc_\mcB)$ if and only if $(\mcA, \oc_\mcA) \preceq_{\mcL_c} (\mcB, \oc_\mcB)$ and moreover $\oc_\mcB(b) = 0$ for every $b \in (B \cap \cl_\mcB(A)) \setminus A$.
Recall that $(\mcA, \oc_\mcA) \preceq_{\mcL_c} (\mcB, \oc_\mcB)$ means that $\mcA \preceq_\mcL \mcB$ and $\oc_\mcB \upharpoonright A = \oc_\mcA$.
Finally, we let $\mcF^+_c$, the set of positive colored $\mcF$-structures, be the set of $(\mcA, \oc_\mcA) \in \mcF_c$ with $(\emptyset, \emptyset) \preceq^* (\mcA, \oc_\mcA)$.
Equivalently, a colored $\mcF$-structure $(\mcA, \oc_\mcA)$ is positive if $\oc_\mcA(a) = 0$ for every $a \in \cl_\mcA(\emptyset)$.

There's a decent number of things we need to check before we can claim that $\preceq^*$ is indeed a notion of strong substructure as in Definition \ref{def:strong_substructure}.
They're all rather straightforward and even a bit tedious to check, but it seems worthwhile to do so in full detail.

\begin{claim}
The relation $\preceq^*$ is a notion of strong substructure on $\mcF_c^+$.
\end{claim}

\begin{proof}
We have a rather large number of conditions from Definition \ref{def:strong_substructure} to check.

For [Preorder], it's trivially reflexive.
To see that it is transitive, suppose we have $(\mcA, \oc_\mcA) \preceq^* (\mcB, \oc_\mcB)$ and $(\mcB, \oc_\mcB) \preceq^* (\mcC, \oc_\mcC)$.
We want to show $(\mcA, \oc_\mcA) \preceq^* (\mcC, \oc_\mcC)$.
We indeed have $(\mcA, \oc_\mcA) \preceq_{\mcL_c} (\mcC, \oc_\mcC)$, so we fix an arbitrary $c \in C \setminus A$ with $c \in \cl_\mcC(A)$.
We need to see that $\oc_\mcC(c) = 0$.
We case on if $c \in B$.
Indeed, if $c \not\in B$ then we have $c \in C \setminus B$ and so from $(\mcB, \oc_\mcB) \preceq^* (\mcC, \oc_\mcC)$ we conclude $\oc_\mcC(c) = 0$.
Otherwise, if $c \in B$ then we have $c \in B \setminus A$ and since $\cl_\mcB(A) = \cl_\mcC(A) \cap B$ by Claim \ref{claim:little_closure}.(2) we have $c \in \cl_\mcB(A)$ and so from $(\mcA, \oc_\mcA) \preceq^* (\mcB, \oc_\mcB)$ we conclude that $\oc_\mcB(c) = 0$ and from $(\mcB, \oc_\mcB) \preceq^* (\mcC, \oc_\mcC)$ we conclude $\oc_\mcC(c) = \oc_\mcB(c) = 0$.

The property [Coarsening] follows directly by the definition of $\preceq^*$.

For [Invariance], let arbitrary $(\mcA, \oc_\mcA) \preceq^* (\mcB, \oc_\mcB)$ and suppose $\pi : (\mcB, \oc_\mcB) \isom (\mcC, \oc_\mcC)$.
Define $(\mcA', \oc_\mcA') := \pi[(\mcA, \oc_\mcA)]$ and so $(\mcA', \oc_{\mcA'}) \preceq_{\mcL_c} (\mcC, \oc_\mcC)$.
We want to show $(\mcA', \oc_\mcA') \preceq^* (\mcC, \oc_\mcC)$.
Indeed, let $c \in C \setminus A'$ with $c \in \cl_\mcC(A')$.
Then $\pi^{-1}(c) \in B \setminus A$ and by Claim \ref{claim:little_closure}.(1) we have $\pi^{-1}(c) \in \cl_\mcB(A)$.
Thus $\oc_\mcB(\pi^{-1}(c)) = 0$ and so $\oc_\mcC(c) = 0$.

For [Downward closure], suppose $(\mcA, \oc_\mcA) \preceq^* (\mcB, \oc_\mcB)$ and $(\mcA, \oc_\mcA) \preceq_{\mcL_c} (\mcC, \oc_\mcC) \preceq_{\mcL_c} (\mcB, \oc_\mcB)$.
We want to show $(\mcA, \oc_\mcA) \preceq^* (\mcC, \oc_\mcC)$.
Indeed, let $c \in C \setminus A$ with $c \in \cl_\mcC(A)$.
Then in fact we have $c \in B \setminus A$ and since by Claim \ref{claim:little_closure}.(2) we have $\cl_\mcC(A) = \cl_\mcB(A) \cap C$ we then have $c \in \cl_\mcB(A)$.
Thus from $(\mcA, \oc_\mcA) \preceq^* (\mcB, \oc_\mcB)$ we conclude $\oc_\mcB(c) = 0$ and so we conclude $\oc_\mcC(c) = 0$.

The property [Positivity] immediately follows by our restriction to $\mcF_c^+$.

Finally we consider [No maximal models].
Let $(\mcA, \oc_\mcA) \in \mcF_c^+$.
We want to see that this is not maximal with respect to $\preceq^*$.
Indeed since $\mcF$ does not have maximal models with respect to $\preceq_\mcL$ we can find $\mcB \in \mcF$ with $\mcA \prec_\mcL \mcB$.
We can then define $\oc_\mcB : B \rightarrow \omega$ by $\oc_{\mcB}(a) = \oc_{\mcA}(a)$ for all $a \in A$ and $\oc_{\mcB}(b) = 0$ for all $b \in B \setminus A$.
Then we get $(\mcB, \oc_\mcB) \in \mcF_c$ and $(\mcA, \oc_\mcA) \preceq^* (\mcB, \oc_\mcB)$.
Since $(\mcA, \oc_\mcA) \in \mcF_c^+$ we have $\emptyset \preceq^* (\mcA, \oc_\mcA)$ and so by transitivity of $\preceq^*$ we get $(\mcB, \oc_\mcB) \in \mcF_c^+$ by definition.

This concludes the proof.
\end{proof}

We have one more thing to check before we can compute the limit of $\mcF^+_c$.

\begin{claim}
$(\mcF^+_c, \preceq^*)$ satisfies the amalgamation property on strong substructures.
\end{claim}

\begin{proof}
Suppose $(\mcA, \oc_{\mcA}), (\mcB, \oc_{\mcB}), (\mcC, \oc_{\mcC}) \in \mcF^+_c$ satisfy $(\mcC, \oc_{\mcC}) \preceq^* (\mcA, \oc_{\mcA})$ and $(\mcC, \oc_{\mcC}) \preceq^* (\mcB, \oc_{\mcB})$.

Since $\mcF$ satisfies the amalgamation property and $\mcM$ is a limit of $\mcF$, we may assume without loss of generality that $\mcA, \mcB \prec_\mcL \mcM$.
Because $\cl$ is disjointifying, we may find some $\mcA' \prec_\mcL \mcM$ with $\mcA' \isom_C \mcA$ as witnessed by $\pi \in \Aut(\mcM)$ such that $\cl(A') \cap B \subseteq \cl(C)$ and $A' \cap \cl(B) \subseteq \cl(C)$ by invariance.
Expand $\mcA'$ to a $\mcL_c$-structure $(\mcA', \oc_\mcA')$ by defining $\oc_\mcA' := \oc_\mcA \circ \pi$, where we can verify $(\mcA', \oc_\mcA') \isom_C (\mcA, \oc_\mcA)$ and $(\mcC, \oc_\mcC) \preceq^* (\mcA', \oc_\mcA')$.

Now define $\mcD := \mcM \upharpoonright (A' \cup B)$.
Then we have $\mcD \in \mcF$ and $\mcA' \preceq_\mcL \mcD$ and $\mcB \preceq_\mcL \mcD$.
Our task is to define the appropriate coloring on $\mcD$.
Indeed, define for each $d \in D$,
\[ \oc_{\mcD}(d) = \begin{cases}
    \oc_{\mcA'}(d) & d \in A' \\
    \oc_{\mcB}(d) & d \in B.
\end{cases}
\]
To see that it is actually well-defined, observe that if $d \in A' \cap B$, then by choice of $A'$ since $A' \subseteq \cl(A')$ we have $d \in \cl(C)$.
If $d \in C$, then of course we have $\oc_{\mcA'}(d) = \oc_{\mcC}(d) = \oc_{\mcB}(d)$.
Otherwise we have $d \in \cl(C) \setminus C$ and so $d \in \cl_{\mcA'}(C) \setminus C$ and $d \in \cl_\mcB(C) \setminus C$ in which by definition of $\preceq^*$ we have $\oc_{\mcA'}(d) = 0 = \oc_{\mcB}(d)$.

We next see that $(\mcA', \oc_\mcA') \preceq^* (\mcD, \oc_\mcD)$.
Indeed, suppose $d \in D \setminus A'$ with $d \in \cl_\mcD(A')$ and thus $d \in \cl(A')$.
Since $D = A' \cup B$, we must have $d \in B$.
By our choice of $A'$ we have $\cl(A') \cap B \subseteq \cl(C)$ thus we have $d \in \cl(C)$.
In particular, we have $d \in \cl_\mcB(C) \setminus C$, and thus from $(\mcC, \oc_\mcC) \preceq^* (\mcB, \oc_\mcB)$ we can conclude $\oc_\mcB(d) = 0$ and thus by definition $\oc_\mcD(d) = 0$ as desired.
The fact that $(\mcB, \oc_\mcB) \preceq^* (\mcD, \oc_\mcD)$ is more of the same.
%Moving on to the general case, we now drop the assumptions that $\mcA, \mcB \prec_\mcL \mcM$.
%As $\mcF$ satisfies the amalgamation property we can find $\mcA' \in \mcF$ such that $\mcA' \isom_\mcC \mcA$ as witnessed by some $\pi$ and $\mcD \in \mcF$ with $\mcA' \preceq_\mcL \mcD$ and $\mcB \preceq_\mcL \mcD$.
%Since $\mcF$ is the age of $\mcM$ we can find some $\hat{\mcD} \in \mcF$ with $\hat{\mcD} \prec_\mcL \mcM$ and $\hat{\mcD} \isom \mcD$ as witnessed by $\sigma$.
%Now define $\hat{\mcA} := \sigma^{-1}[\mcA']$ and $\hat{\oc}_\mcA := \oc_\mcA \circ \pi \circ \sigma$ and $\hat{\mcB} := \sigma^{-1}[\mcB]$ and $\hat{\oc}_\mcB := \oc_\mcB \circ \sigma$ and $\hat{\mcC} := \sigma^{-1}[\mcC]$ and $\hat{\oc}_\mcC := \oc_\mcC \circ \sigma$.
%It's straightforward to verify $(\hat{\mcA}, \hat{\oc}_\mcA)$, $(\hat{\mcB}, \hat{\oc}_\mcB)$, $(\hat{\mcC}, \hat{\oc}_\mcC) \in \mcF^+_c$ and $(\hat{\mcC}, \hat{\oc}_\mcC) \preceq^* (\hat{\mcA}, \hat{\oc}_\mcA)$ and $(\hat{\mcC}, \hat{\oc}_\mcC) \preceq^* (\hat{\mcB}, \hat{\oc}_\mcB)$ and $\hat{\mcA}, \hat{\mcB} \prec_\mcL \mcM$.
\end{proof}

By Proposition \ref{prop:exists_limit} there is a countable limit $(\mcN, \oc)$ of $(\mcF^+_c, \preceq^*)$.

\begin{claim}
We have $\mcM \isom \mcN$.
\end{claim}

\begin{proof}
It suffices by a back-and-forth argument that the set
\[ \mcI = \{\pi : \mcA \isom \mcB \mid \mcA \prec_\mcL \mcM \; \text{and} \; (\mcB, \oc_\mcB) \prec^* (\mcN, \oc) \; \text{where} \; \oc_\mcB = \oc \upharpoonright B\}.\]
has the back-and-forth property.

First, suppose we have $\pi : \mcA \isom \mcB \in \mcI$ and some $a' \in M$.
Define $\mcA' := \mcM \upharpoonright Aa'$, and extend $\pi$ arbitrarily to $\pi' : \mcA' \isom \mcB'$ where $B' = Bb'$ for some $b' \not\in B$.
Define $\oc_\mcB'$ on $B'$ by $\oc_\mcB'(b) = \oc(b)$ for all $b \in B$ and $\oc_{\mcB'}(b') = 0$.
Then we have $(\mcB, \oc_\mcB) \prec^* (\mcB', \oc_\mcB')$.
By Lemma \ref{lem:extension_property}, there is some $\sigma :  (\mcB'', \oc_\mcB'') \isom_B (\mcB', \oc_\mcB')$ with $(\mcB'', \oc_\mcB'') \prec^* (\mcN, \oc)$.
By definition of $\prec^*$, we must have $\oc_\mcB'' = \oc \upharpoonright B''$ and thus $\sigma^{-1} \circ \pi' : \mcA' \isom \mcB''$ is in $\mcI$.
We moreover have $\pi \subseteq \sigma^{-1} \circ \pi'$ and so the first part is done.

For the other half, suppose $\pi : \mcA \isom \mcB \in \mcI$ and $b' \in N$.
Applying Definition \ref{def:strong_limit}.(i) there is some $(\mcB', \oc_\mcB') \prec^* (\mcN, \oc)$ with $Bb' \subseteq B'$.
By definition of $\prec^*$ we have $\oc_\mcB' = \oc \upharpoonright B'$.
Extend $\pi$ arbitrarily to some $\pi' : \mcA' \isom \mcB'$ where $\mcA \prec_\mcL \mcA'$.
By the extension property for $\preceq_\mcL$ on $\mcF$, find some $\sigma : \mcA'' \isom_A \mcA'$ such that $\mcA'' \prec_\mcL \mcM$.
Then $\pi' \circ \sigma : \mcA'' \isom \mcB'$ is in $\mcI$ and extends $\pi$ as desired.
\end{proof}

Thus, without loss of generality we will assume $\mcM = \mcN$ and thus we have equipped $\mcM$ with a coloring $\oc$ such that $(\mcM, \oc)$ is a strong limit of $(\mcF^+_c, \preceq^*)$.

Let $S_\infty^*$ be the closed subgroup of $S_\infty$ of permutations of $\omega$ that fix 0.
This is easily isomorphic to $S_\infty$, thus it suffices to show that $\Aut(\mcM)$ involves $S_\infty^*$.

Let $H \le \Aut(\mcM)$ be the closed subgroup of $\pi$ satisfying that for every $a, b \in M$ we have $\oc(a) = \oc(b)$ if and only if $\oc(\pi(a)) = \oc(\pi(b))$.
Observe that $\pi \in H$ if and only if there is some $\sigma_\pi \in S^*_\infty$ such that $\oc(\pi(a)) = \sigma_\pi(\oc(a))$ for every $a \in M$.
Because $\cl$ is nontrivial, every color appears somewhere in $(\mcM, \oc)$ by Definition \ref{def:strong_limit}.(ii).
The map $\pi \mapsto \sigma_\pi$ is therefore well defined, and moreover is easily seen to be a continuous homomorphism from $H$ to $S_\infty^*$.
What remains is to show that this map is surjective.

Fix some arbitrary permutation $\sigma \in S_\infty^*$.
Let $\{d_n \mid n \in \omega\}$ be an enumeration of $M$.
We will define $(\mcA_n, \oc_{\mcA_n})$ and $(\mcB_n, \oc_{\mcB_n})$ in $\mcF^+_c$ satisfying for every $n$:
\begin{enumerate}
    \item $(\mcA_n, \oc_{\mcA_n}) \prec^* (\mcM, \oc)$ and $\mcA_n \preceq_\mcL \mcA_{n+1}$ and $(\mcB_n, \oc_{\mcB_n}) \prec^* (\mcM, \oc)$ and $\mcB_n \preceq_\mcL \mcB_{n+1}$;
    \item $d_n \in A_{2n+1}$ and $d_n \in B_{2n+2}$;
    \item $\pi_n[A_n] = B_n$;
    \item $\oc_{\mcB_n}(\pi_n(a)) = \sigma(\oc_{\mcA_n}(a))$ for every $a \in A_n$;
    \item $\pi_{n+1} \upharpoonright A_n = \pi_n$. 
\end{enumerate}

Assuming we are able to produce this, we claim that the sequence of $\pi_n$ is Cauchy.
Recall that $d(\pi, \pi') := 2^{-n-1}$ where $n$ is the least such that $\pi(d_n) \neq \pi'(d_n)$ is a compatible left-invariant metric on $\Aut(\mcM)$ and $D(\pi, \pi') := d(\pi, \pi') + d(\pi^{-1}, \pi'^{-1})$ is a compatible complete metric on $\Aut(\mcM)$, see e.g. the beginning of \cite[1.5]{BeckerKechris1996}.
For any $m$, if we choose $n, n' \ge 2m+2$ with $n < n'$ then we have $\{d_i \mid i \le m\} \subseteq A_n$ by (2) and (1) and thus $d(\pi_n, \pi_{n'}) < 2^{-m-1}$ by (5).
Moreover we have $\{d_i \mid i \le m\} \subseteq B_n$ by (2) and (1) and so $d(\pi^{-1}_n, \pi^{-1}_{n'}) <2^{-m-1}$ by (3).
Thus $D(\pi_n, \pi_{n'}) < 2^{-m}$.
Thus by completeness of $\Aut(\mcM)$ we have $\pi_n \rightarrow \pi_\infty$ for some $\pi_\infty \in \Aut(\mcM)$.

Arguing similarly, we have that $\oc(\pi_\infty(a)) = \sigma(\oc(a))$ for every $a \in M$.
In particular, $\pi_\infty \in H$ and $\sigma_{\pi_\infty} = \sigma$.

Thus all that remains is the construction.
Let $A_0 = B_0 = \emptyset$.

Suppose for some $n$ we've defined $(\mcA_i)_{i \le 2n}$ and $(\mcB_i)_{i \le 2n}$ and $(\pi_i)_{i \le 2n}$.

By Definition \ref{def:strong_limit}.(i) of strong limit, we may find some $(\mcA_{2n+1}, \oc_{\mcA_{2n+1}}) \in \mcF^+_c$ with $A_{2n}\cup\{d_n\} \subseteq A_{2n+1}$ and $(\mcA_{2n+1}, \oc_{\mcA_{2n+1}}) \prec^* (\mcM, \oc)$ and so in particular 
\begin{equation}\label{eq:A_strong}
(\mcA_{2n}, \oc_{\mcA_{2n}}) \prec^* (\mcA_{2n+1}, \oc_{\mcA_{2n+1}}).
\end{equation}
Now let $(\hat{\mcB}, \hat{\oc})$ be defined by
\[\hat{\mcB} := \mcM \upharpoonright \pi_{2n}[A_{2n+1}] \quad \text{and} \quad \hat{\oc} := \sigma \circ \oc \circ \pi_{2n}^{-1}.\]
Observe that 
\begin{equation}\label{eq:B_strong}
(\mcB_{2n}, \oc_{\mcB_{2n}}) \prec^* (\hat{\mcB}, \hat{\oc}).
\end{equation}
Indeed, checking all parts of the definition of $\prec^*$, we have $\mcB_{2n} \prec_\mcL \hat{\mcB}$, and for every $b \in B_{2n}$ we have $\hat{\oc}(b) = \sigma(\oc(\pi_{2n}^{-1}(b))) = \oc(\pi_{2n}(\pi_{2n}^{-1}(b))) = \oc(b)$ where the second equality is by (4), and finally for every $b \in \cl_{\hat{\mcB}}(B_{2n}) \setminus B_{2n}$ we have by invariance Claim \ref{claim:little_closure}.(i) that $\pi_{2n}^{-1}(b) \in \cl_{\mcA_{2n+1}}(A_{2n}) \setminus A_{2n}$ and so $\hat{\oc}(b) = \sigma(\oc(\pi_{2n}^{-1}(b))) = \sigma(0) = 0$ by Equation \ref{eq:A_strong} and that $\sigma \in S^*_\infty$.

By Lemma \ref{lem:extension_property}, there is some $(\mcB_{2n+1}, \oc_{\mcB_{2n+1}}) \in \mcF^+_c$ such that
\begin{equation}\label{eq:B_strong2}
(\mcB_{2n+1}, \oc_{\mcB_{2n+1}}) \prec^* (\mcM, \oc) \quad \text{and} \quad (\mcB_{2n+1}, \oc_{\mcB_{2n+1}}) \isom_{B_{2n}} (\hat{\mcB}, \hat{\oc}).
\end{equation}
Letting $\rho$ witness the isomorphism on the right, we apply ultrahomogeneity of strong substructures Definition \ref{def:strong_limit}.(iii) to find $\pi \in \Aut(\mcM, \oc)$ extending $\rho$.
Then define $\pi_{2n+1} := \pi^{-1} \circ \pi_{2n}$.

Let's check that we satisfied all the properties up to $2n+1$.
We can see that we've satisfied (1).
We satisfied (2) by our choice of $A_{2n+1}$.
For (3) we observe $\pi_{2n+1}[\mcA_{2n+1}] = (\pi^{-1} \circ \pi_{2n})[\mcA_{2n+1}] = \pi^{-1}[\hat{\mcB}] = \mcB_{2n+1}$.
For (4) for any $a \in A_{2n+1}$ we see that 
\begin{align*}
\oc_{\mcB_{2n+1}}(\pi_{2n+1}(a)) &= \oc_{\mcB_{2n+1}}(\pi^{-1}(\pi_{2n}(a))) = \hat{\oc}(\pi_{2n}(a)) \\
&= (\sigma \circ \oc \circ \pi_{2n}^{-1})(\pi_{2n}(a)) = \sigma(\oc(a)) = \sigma(\oc_{\mcA_{2n+1}}(a)).
\end{align*}
Finally for (5) we see that for any $a \in A_{2n}$ we have $\pi_{2n+1}(a) = \pi^{-1}(\pi_{2n}(a)) = \pi_{2n}(a)$ since $\pi_{2n}(a) \in B_{2n}$ and $\pi$ extends $\rho$.

The construction of $\mcA_{2n+2}$ and $\mcB_{2n+2}$ and $\pi_{2n+2}$ are similar.

Thus we have showed that $P = \Aut(\mcM)$ involves $S_\infty$.
This completes the proof of Theorem \ref{thm:disjointifying_closure_implies_involve}.
\end{proof}

\section{A rank function and the minimal disjointifying closure operator}\label{sec:rank}

We now identify a rank function which identifies the existence of a disjointifying closure operator. It will also allow us to define a canonical closure operator which will happen to be the minimal disjointifying closure operator if one exists.
The rank that we will define will be a natural extension of a rank function which is already in the literature.
We will start by discussing this rank function and how it characterizes the pseudo-algebraic closure, and then proceed to defining our new rank function and the analogous way in which it characterizes the minimal disjointifying closure operator.

\subsection{Deissler rank}

This subsection is a presentation of the theory developed by Deissler in \cite{Deissler1997} in somewhat different language.

Let $I$ be a set with an action $P \curvearrowright I$.
For any finite $B \subseteq I$ let $\Stab_P(B)$ be the set of $g \in P$ such that $g \cdot b = b$ for all $b \in B$.
Given a finite set $B \subseteq I$ and $a \in I$, we define an ordinal rank $\Drk(a, B)$ as follows.
We say $\Drk(a, B) \le 0$ iff for every $g \in \Stab_P(B)$, we have $g \cdot a = a$.
In general, for $\alpha > 0$, we say $\Drk(a, B) \le \alpha$ iff there exists some $c \in I$ such that for every $c' \isom_B c$, we have $\Drk(a, Bc') < \alpha$.
We then define $\Drk(a,B)$ to be the least $\alpha$ such that $\Drk(a, B) \le \alpha$, if such $\alpha$ exists, otherwise we write $\Drk(a, B) = \infty$.
If $\Drk(a, B) \le \alpha$ for some ordinal $\alpha$, then we write $\Drk(a, B) < \infty$.

The following facts are useful and easily verified by the most straightforward induction arguments.

\begin{lemma}\label{lem:deissler_rank_fact}
    The following all hold:
    \begin{enumerate}
    \item For finite subsets $B, C$ of $I$ satisfying $C \supseteq B$ and $a \in I$, we have $\Drk(a, C) \le \Drk(a, B)$;
    \item For any finite $B \subseteq I$ and $a \in I$, if there exists some $c \in I$ such that $\Drk(a, c'B) < \infty$ for every $c' \isom_B c$, then $\Drk(a, B) < \infty$;
    \item For finite $B \subseteq I$ and $a \in I$ if $g \in P$ then $\Drk(g \cdot a, g \cdot B) = \Drk(a, B)$.
    \end{enumerate}
\end{lemma}

We define a closure operator $\Dcl$ on $I$ as follows: given $B \subseteq I$ and $a \in I$, we declare $a \in \Dcl(B)$ iff $\Drk(a, B_0) < \infty$ for some finite $B_0 \subseteq B$.
The fact that it is actually a closure operator follows by the next result.

\begin{theorem}[ess. Deissler]\label{th:deissler}
Let $\mcL$ be a countable relational language and $\mcM$ a countable $\mcL$-structure with the natural action $\Aut(\mcM) \curvearrowright M$. 
Let $B \subseteq M$ be finite and $a \in M$.
Then $a \in \Dcl(B)$ iff for every $\mcM_0 \preceq_{\mcL_{\omega_1, \omega}} \mcM$ with $B \subseteq M_0$, we have $a \in M_0$.
\end{theorem}

\begin{proof}
Recall from Lemma \ref{lem:elementary_substructure} that for any $\mcL$-substructure $\mcM_0 \preceq_{\mcL} \mcM$, we have $\mcM_0 \prec_{\countlogic} \mcM$ iff for every finite $B \subseteq M_0$ and $a \in M$, there is some $a' \in M_0$ with $a' \isom_B a$.

For the forward direction, suppose $a \in \Dcl(B)$ and $\mcM_0 \prec_{\countlogic} \mcM$ with $B \subseteq M_0$.
We argue by induction on $\Drk(a, B)$.
If $\Drk(a, B) \le 0$ and $a' \isom_B a$ satisfies $a' \in M_0$, then we must have $a' = a$ and thus $a \in M_0$.
Otherwise, suppose $\Drk(a, B) \le \alpha$ with $\alpha > 0$ and the claim is true below $\alpha$.
Fix some $c \in M$ such that $\Drk(a, Bc') < \alpha$ for every $c' \isom_B c$.
Let $c' \isom_B c$ such that $c' \in M_0$ in which case by the induction hypothesis we have $a \in M_0$.

For the backwards direction, suppose $a \not\in \Dcl(B)$.
We will construct a $\countlogic$-elementary submodel $\mcM_0$ of $\mcM$ with $a \not\in M_0$.

Given $C \subseteq C' \subseteq M$, say that $C'$ has closure property $(*)$ over $C$ iff for every finite $D \subseteq C$ and $e \in M$, there exists some $e' \isom_D e$ with $e' \in C'$.
We argue that for any $C \subseteq M$ with $a \not\in \Dcl(C)$, there is a set $C' \subseteq M$ with $C \subseteq C'$ which has closure property $(*)$ over $C$ and $a \not\in \Dcl(C')$.
Granting this, we can define a sequence 
\[\emptyset = C_0 \subseteq C_1 \subseteq C_2 \subseteq ...\]
such that $a \not\in \Dcl(C_n)$ and $C_{n+1}$ has property $(*)$ over $C_n$, and define $\mcM_0 = \mcM \upharpoonright [\bigcup_{n \in \omega} C_n]$.
Then $a \not\in M_0$ and $\mcM_0 \preceq_{\countlogic} \mcM$.

Let $C \subseteq M$.
To find $C'$, write $C$ as an increasing union $\bigcup_{n \in \omega} B_n$ of finite sets.
Fix an enumeration $\{a_n \mid n \in \omega\}$ of $M$.
Let $h : \omega \rightarrow \omega$ be a function such that the preimage of every $n \in \omega$ is infinite.
We recursively define using Lemma \ref{lem:deissler_rank_fact} a sequence $d_0, d_1, ...$ such that for every $n \in \omega$, $d_n \isom_{B_n} a_{h(n)}$ and $a \not\in \Dcl(B_nd_0...d_{n})$.
Then let $C' = C \cup \{d_n \mid n \in \omega\}$.
It is easy to check that this works.
\end{proof}

\begin{proposition}\label{prop:Dcl_property}
The operator $\Dcl$ is an invariant closure operator with finite character.
\end{proposition}

\begin{proof}
The fact that it is a closure operator is a direct corollary of Theorem \ref{th:deissler}.
Invariance follows from Lemma \ref{lem:deissler_rank_fact}.(3).
It has finite character by definition.
\end{proof}

In \cite{Gao1998}, Gao relates the non-existence of nontrivial $\mcL_{\omega_1, \omega}$-elementary substructures with dynamical properties of the automorphism group.
Recall that a Polish group $G$ is cli iff there is a complete metric $d$ on $G$, compatible with the topology on $G$, such that $d$ is left-invariant, i.e. $d(g, g') = d(hg, hg')$ for every $g, g', h \in G$.

\begin{theorem}[Gao, \cite{Gao1998}]
Let $\mcL$ be a countable relational language and $\mcM$ a countable $\mcL$-structure. 
Then the Polish group $\Aut(\mcM)$ is cli iff there is no $\mcM_0 \prec_{\countlogic} \mcM$, i.e. there is no non-trivial $\countlogic$-elementary substructure of $\mcM$.
\end{theorem}

Putting these equivalences together, we get the following ``master" list of equivalences, where we define $\Drk(\mcM)$ to be the supremum of $\Drk(a, \emptyset) + 1$ over all $a \in M$ with respect to the natural action $\Aut(\mcM) \curvearrowright M$.

\begin{corollary}[Deissler, Gao]\label{cor:deissler}
Let $\mcL$ be a countable relational language and $\mcM$ a countable $\mcL$-structure.
The following are equivalent:
\begin{enumerate}
    \item $\Aut(\mcM)$ is cli;
    \item there is no $\mcM_0 \prec_{\countlogic} \mcM$;
    \item there is no uncountable $\mcL$-structure satisfying the Scott sentence of $\mcM$;
    \item $\Dcl(\emptyset) = M$;
    \item $\Drk(\mcM) < \infty$;
    \item $\Drk(\mcM) < \omega_1$.
\end{enumerate}
\end{corollary}
The equivalence of (5) and (6) is easily seen by induction on $\alpha$ that for any $a \in M$ and $B \subseteq M$ if $\Drk(a, B) \le \alpha$ then $\Drk(a, B) < \omega_1$.

\subsection{Disjointifying rank}
With the Deissler rank as motivation, we move on to the disjointifying rank.
Let $I$ be a set with a group action $P \curvearrowright I$.
Given a finite set $B \subseteq I$ and $a \in I$, we define an ordinal rank $\Krk(a, B)$ as follows.
We say $\Krk(a, B) \le 0$ iff for every $g \in \Stab_P(B)$, we have $g \cdot a = a$.
In general for $\alpha > 0$, we say $\Krk(a, B) \le \alpha$ iff at least one of the following holds:
\begin{enumerate}
\item[(a)] there exists some $c$ such that for every $c' \isom_B c$, we have $\Krk(a, Bc') < \alpha$;
or
\item[(b)] for every $a' \isom_B a$, either $\Krk(a, Ba') < \alpha$ or $\Krk(a', Ba) < \alpha$.
\end{enumerate}

We then define $\Krk(a, B)$ to be the least $\alpha$ such that $\Krk(a, B) \le \alpha$, if such $\alpha$ exists, otherwise we write $\Krk(a, B) = \infty$.
If $\Krk(a, B) \le \alpha$ for some ordinal $\alpha$, then we write $\Krk(a, B) < \infty$.

Note that if we were to remove condition (b) in the recursive case of the definition of disjointifying rank, we will just get the Deissler rank.

We list some basic properties of the rank which are easily confirmed by the most straightforward induction argument.
\begin{lemma}\label{lem:rank_basic_facts} The following all hold:
\begin{enumerate}
    \item[(1)]\label{item:knight_monotone} For finite subsets $B, C$ of $I$ satisfying $C \supseteq B$, and $a \in I$, we have $\Krk(a, C) \le \Krk(a, B)$;
    \item[(2a)]\label{item:knight2a} For finite $B \subseteq I$ and $a \in I$, if there exists some $c \in I$ such that $\Krk(a, c'B) < \infty$ for every $c' \isom_B c$, then $\Krk(a, B) < \infty$;
    \item[(2b)]\label{item:knight2b} For finite $B \subseteq I$ and $a \in I$, if either $\Krk(a, a'B) < \infty$ or $\Krk(a', aB) < \infty$ for every $a' \isom_B a$, then $\Krk(a, B) < \infty$.
    \item[(3)]\label{item:knight_invariant} For finite $B \subseteq I$ and $a \in I$ if $g \in P$ then $\Krk(g \cdot a, g \cdot B) = \Krk(a, B)$.
\end{enumerate}
\end{lemma}

The following properties are also useful but demand proofs.

\begin{lemma}\label{lem:clmin_closure}
The following hold for all $x, y \in I$ and finite $X, Y, Z \subseteq I$:
\begin{enumerate}
\item If $\Krk(x, Z) \le 0$ and $\Krk(y, xZ) < \infty$ then $\Krk(y, Z) < \infty$;
\item If $\Krk(x, Z) < \infty$ and $\Krk(y, xZ) < \infty$ then $\Krk(y, Z) < \infty$; 
\item If $\Krk(x, Z) < \infty$ for all $x \in X$ and $\Krk(y, XZ) < \infty$ then $\Krk(y, Z) < \infty$
\end{enumerate}
\end{lemma}

\begin{proof}
We start with Statement (1), inducting on $\alpha := \Krk(y, xZ)$.
We start with the base case $\alpha = 0$.
Letting $g \in \Stab_P(Z)$ by $\Krk(x, Z) \le 0$ we have $g \cdot x = x$ and so $g \in \Stab_P(xZ)$ and so by $\Krk(y, xZ) \le 0$ we have $g \cdot y = y$. 
Thus $\Krk(y, Z) \le 0$.

Otherwise assume $\alpha > 0$ and the claim holds below $\alpha$.
In the inductive definition of $\Krk(y, xZ)$ there are two cases, which we handle separately.

First, suppose there is some $w$ such that for every $w' \isom_{xZ} w$ we have $\Krk(y, w'xZ) < \alpha$.
To show $\Krk(y, Z) < \infty$ by Lemma \ref{lem:rank_basic_facts}.(2a) it suffices to show $\Krk(y, w'Z) < \infty$ for every $w' \isom_Z w$.
Indeed, fix arbitrary $w' \isom_{Z} w$ as witnessed by $g \in P$.
Since $\Krk(x, Z) \le 0$ we must have $g \cdot x = x$.
Thus $g : w' \isom_{xZ} w$ and so $\Krk(y, w'xZ) < \alpha$.
By monotonicity Lemma \ref{lem:rank_basic_facts}.(1) we have $\Krk(x, w'Z) \le 0$, and so by the induction hypothesis we have $\Krk(y, w'Z) < \infty$ as desired.

Second, suppose for every $y' \isom_{xZ} y$ either $\Krk(y', yxZ) < \alpha$ or $\Krk(y, y'xZ) < \alpha$.
To show $\Krk(y, Z) < \infty$ it suffices to show by Lemma \ref{lem:rank_basic_facts}.(2b) that for every $y' \isom_Z y$ either $\Krk(y', yZ) < \infty$ or $\Krk(y, y'Z) < \infty$.
Indeed, fix arbitrary $y' \isom_Z y$ as witnessed by $g$.
Since $\Krk(x, Z) \le 0$ we in fact have $g : y' \isom_{xZ} y$.
If $\Krk(y', yxZ) < \alpha$ then by the induction hypothesis we get $\Krk(y', yZ) < \infty$.
If $\Krk(y, y'xZ) < \alpha$ then by the induction hypothesis we get $\Krk(y, y'Z) < \infty$.

Now we move on to Statement (2).
We induct on $\alpha := \Krk(x, Z)$.
If $\alpha = 0$ then this is just Statement (1), thus we assume $\alpha > 0$ and the claim holds below $\alpha$.
In the inductive definition of $\Krk(x, Z)$ there are two cases, which we handle separately.

First, suppose there is some $w$ such that for every $w' \isom_Z w$ we have $\Krk(x, w'Z) < \alpha$.
To show $\Krk(y, Z) < \infty$ it suffices to show by Lemma \ref{lem:rank_basic_facts}.(2a) that for arbitrary $w' \isom_Z w$ we have $\Krk(y, w'Z) < \infty$.
We have $\Krk(x, w'Z) < \alpha$.
Since $\Krk(y, xZ) < \infty$ we have by Lemma \ref{lem:rank_basic_facts}.(1) that $\Krk(y, xw'Z) < \infty$ and so by the induction hypothesis we have $\Krk(y, w'Z) < \infty$.

Second, suppose for every $x' \isom_Z x$ either $\Krk(x', xZ) < \alpha$ or $\Krk(x, x'Z) < \alpha$.
For any $y' \isom_Z y$ and $y'' \isom_Z y$ write $y' \le y''$ iff for every $x' \isom_Z x$ if $\Krk(y', x'Z) < \infty$ then $\Krk(y'', x'Z) < \infty$.
We claim this is a prelinear order.
It is clearly reflexive and transitive, so we move on to trichotomy.
Indeed, suppose for contradiction that we have such $y', y''$ such that $y' \not\le y''$ and $y'' \not\le y'$.
Then we have $x' \isom_Z x$ and $x'' \isom_Z x$ such that
\begin{equation}\label{eq:trich_1}
\Krk(y', x'Z) < \infty \quad \text{and} \quad \Krk(y'', x'Z) = \infty
\end{equation}
\begin{equation}\label{eq:trich_2}
\Krk(y', x''Z) = \infty \quad \text{and} \quad \Krk(y'', x''Z) < \infty.
\end{equation}

By invariance Lemma \ref{lem:rank_basic_facts}.(3), we have either $\Krk(x', x''Z) < \alpha$ or $\Krk(x'', x'Z) < \alpha$.
If $\Krk(x', x''Z) < \alpha$ then because the left side of (\ref{eq:trich_1}) implies $\Krk(y', x'x''Z) < \infty$ by monotonicity Lemma \ref{lem:rank_basic_facts}.(1), we can conclude by the induction hypothesis that $\Krk(y', x''Z) < \infty$, but this contradicts the left side of (\ref{eq:trich_2}).
On the other hand, if $\Krk(x'', x'Z) < \alpha$ then because the right side of (\ref{eq:trich_2}) implies $\Krk(y'', x''x'Z) < \infty$ by monotonicity, we can conclude by the induction hypothesis that $\Krk(y'', x'Z) < \infty$, but this contradicts the right side of (\ref{eq:trich_1}).
In either case we get a contradiction, thus $\le$ is a prelinear order.

Now we argue that if $y' \isom_Z y$ and $y'' \isom_Z y$ satisfy $y' \le y''$ then $\Krk(y'', y'Z) < \infty$.
This would imply that $\Krk(y, Z) < \infty$ as desired by Lemma \ref{lem:rank_basic_facts}.(2b).
Indeed, suppose $y' \le y''$ and let $g \in P$ witness $y' \isom_Z y$.
Define $x' := g^{-1} \cdot x$ in which case we get $x' \isom_Z x$.
Thus $\Krk(y', x'Z) = \Krk(y, xZ) < \infty$ by invariance Lemma \ref{lem:rank_basic_facts}.(3).
By Lemma \ref{lem:rank_basic_facts}.(2a) to show $\Krk(y'', y'Z) < \infty$ it suffices to show that for an arbitrary $x'' \isom_{y'Z} x'$, we have $\Krk(y'', x''y'Z) < \infty$.
By monotonicity Lemma \ref{lem:rank_basic_facts}.(1) it suffices to show $\Krk(y'', x''Z) < \infty$.
By definition of $\le$ it suffices to show $\Krk(y', x''Z) < \infty$.
By invariance, since $x'' \isom_{y'Z} x'$ we have $\Krk(y', x''Z) = \Krk(y', x'Z)$ which is less than $\infty$.

We proceed to Statement (3).
We induct on $|X|$.
If $|X| = 0$ then the claim follows immediately so assume $|X| > 0$ and the claim is true below $|X|$.
Fix any $x_0 \in X$ and let $X_0 = X \setminus \{x_0\}$.
Then by monotonicity Lemma \ref{lem:rank_basic_facts}.(1) we have $\Krk(x, Zx_0) < \infty$ for all $x \in X_0$ and $\Krk(y, X_0 Z x_0) < \infty$.
Thus by the induction hypothesis we have $\Krk(y, Z x_0) < \infty$.
We have $\Krk(x_0, Z) < \infty$ and $\Krk(y, Zx_0) < \infty$ so by Statement (2) we have $\Krk(y, Z) < \infty$ as desired.
\end{proof}

We define a closure operator $\clmin$ on $I$ by saying for $a \in I$ and $B \subseteq I$ that $a \in \clmin(B)$ iff for some finite $B_0 \subseteq B$ we have $\Krk(a, B_0) < \infty$.
We need to actually verify that it is a closure operator, which we do in the next lemma.

\begin{lemma}\label{lem:clmin_closure2}
$\clmin$ is a disjointifying closure operator.
\end{lemma}

\begin{proof}
For an arbitrary $A \subseteq I$ we need to show that 
\begin{enumerate}
\item $A \subseteq \clmin(A)$;
\item $\clmin(A) \subseteq \clmin(B)$ whenever $A \subseteq B$; and
\item $\clmin(\clmin(A)) = \clmin(A)$.
\end{enumerate}

For (1), we simply observe that for any $a \in A$ we have $\Krk(a, a) \le 0$ and thus $a \in \clmin(A)$.
For (2), we simply observe that if $c \in \clmin(A)$ then there is some finite $A_0 \subseteq A$ such that $\Krk(c, A_0) < \infty$ and since $A_0 \subseteq B$ we conclude $c \in \clmin(B)$.
For (3), let $c \in \clmin(\clmin(A))$ be arbitrary.
Let $D_0 \subseteq \clmin(A)$ be finite such that $\Krk(c, D_0) < \infty$.
For each $d \in D_0$ there is some finite $A_d \subseteq A$ such that $\Krk(d, A_d) < \infty$.
Letting $A_0 := \bigcup_{d \in D_0} A_d$ we have by monotonicity that $\Krk(d, A_0) < \infty$ for every $d \in D_0$.
We also get by monotonicity Lemma \ref{lem:rank_basic_facts}.(1) that $\Krk(c, D_0A_0) < \infty$.
By Lemma \ref{lem:clmin_closure}.(3) we conclude $\Krk(c, A_0) < \infty$ and thus $c \in \clmin(A)$.

The fact that $\clmin$ is invariant follows by Lemma \ref{lem:rank_basic_facts}.(3) and finite-character follows by definition.

To show that it is disjointifying, by Proposition \ref{prop:cl_equivalence}.(4) it suffices to show that for arbitrary finite $Z \subseteq I$ and $x, y \in I$ with $x \not\in \clmin(Z)$ that 
\begin{enumerate}
\item[(a)] there is some $x' \isom_Z x$ with $x'Z \forkindep[Z] xZ$; and
\item[(b)] there is some $x' \isom_Z x$ with $x' \not\in \clmin(yZ)$
\end{enumerate}
where $\forkindep$ is the independence relation derived from $\clmin$.

We first check (a).
Suppose for contradiction that for every $x' \isom_Z x$ we have $x'Z \not\forkindep[Z] xZ$, or equivalently, either $\clmin(x'Z) \cap xZ \not\subseteq \clmin(Z)$ or $\clmin(xZ) \cap x'Z \not\subseteq \clmin(Z)$.
The first holds if and only if $x \in \clmin(x'Z)$ and the second if and only if $x' \in \clmin(xZ)$.
Thus we have shown for every $x' \isom_Z x$ either $\Krk(x, x'Z) < \infty$ or $\Krk(x', xZ) < \infty$ and thus by Lemma \ref{lem:rank_basic_facts}.(2b) we have $\Krk(x, Z) < \infty$ and so $x \in \clmin(Z)$, a contradiction.

Next we check (b).
Suppose for contradiction that for every $x' \isom_Z x$ we have $x' \in \clmin(yZ)$.
Thus by Lemma \ref{lem:one_side} we have for every $y' \isom_Z y$ that $x \in \clmin(y'Z)$ and thus $\Krk(x, y'Z) < \infty$ and so by Lemma \ref{lem:rank_basic_facts}.(2a) we have $\Krk(x, Z) < \infty$, a contradiction.
Thus we have shown that $\clmin$ is disjointifying.

\end{proof}

\begin{proposition}
$\clmin$ is the minimum invariant disjointifying closure operator with finite character.
\end{proposition}

\begin{proof}
Let $\cl$ be another disjointifying closure operator and let $\forkindep$ be the independence relation derived from $\cl$.
Fixing arbitrary finite $Z \subseteq I$ it suffices to show that $\clmin(Z) \subseteq \cl(Z)$.
Indeed, fix arbitrary $x \in \clmin(Z)$ which means $\Krk(x, Z) < \infty$.
We induct on $\alpha := \Krk(x, Z)$.
For the base case, suppose $\alpha = 0$.
Since $\cl$ is disjointifying by Proposition \ref{prop:cl_equivalence}.(3) we can find some $x' \isom_Z x$ such that $x'Z \forkindep[Z] xZ$.
But since $\Krk(x, Z) \le 0$ we must have $x' = x$.
However $xZ \forkindep[Z] xZ$ can only hold if $x \in \cl(Z)$.

Now let $\alpha > 0$ and assume the claim holds below $\alpha$.
By the inductive definition of $\Krk(x, Z) \le \alpha$ either (a) there is some $y$ such that for every $y' \isom_Z y$ we have $\Krk(x, y'Z) < \alpha$, or (b) for every $x' \isom_Z x$ either $\Krk(x', xZ) < \alpha$ or $\Krk(x, x'Z) < \alpha$.

First suppose we have $y$ as in case (a).
Because $\cl$ is disjointifying find $y' \isom_Z y$ such that $xZ \forkindep[Z] y'Z$, and so in particular $xZ \cap \cl(y'Z) \subseteq \cl(Z)$.
By the induction hypothesis since $\Krk(x, y'Z) < \alpha$ we have $x \in \cl(y'Z)$, and thus $x \in \cl(Z)$.

Next, we suppose we have case (b).
Because $\cl$ is disjointifying find $x' \isom_Z x$ such that $xZ \forkindep[Z] x'Z$, thus \[\cl(xZ) \cap x'Z \subseteq \cl(Z) \quad \text{and} \quad \cl(x'Z) \cap xZ \subseteq \cl(Z).\]
If $\Krk(x', xZ) < \alpha$ then by the induction hypothesis we have $x' \in \cl(xZ)$ and so by the left inclusion we have $x' \in \cl(Z)$ and thus by invariance of $\cl$ we get $x \in \cl(Z)$.
On the other hand if $\Krk(x, x'Z) < \alpha$ then by the induction hypothesis we have $x \in \cl(x'Z)$ and so by the right inclusion we get $x \in \cl(Z)$.
\end{proof}

Given a countable relational language $\mcL$ and $\mcL$-structure $\mcM$ we can define $\Krk(\mcM)$ as before to be the supremum of $\Krk(a, \emptyset) + 1$ as $a$ ranges over $M$, when we consider the natural action $\Aut(\mcM) \curvearrowright M$.

Then mirroring Corollary \ref{cor:deissler} for Deissler rank, we get the following equivalences, though one equivalence we will have to postpone until the next section.

\begin{corollary}\label{cor:equivalences}
Let $\mcL$ be a countable relational language and $\mcM$ a countable $\mcL$-structure. The following are equivalent:
\begin{enumerate}
\item $\Aut(\mcM)$ does not involve $S_\infty$;
\item $\cl(\emptyset) = M$ for every invariant disjointifying closure operator with finite character;
\item $\clmin(\emptyset) = M$;
\item $\Krk(\mcM) < \infty$;
\item $\Krk(\mcM) < \omega_1$.
\end{enumerate}
\end{corollary}

\begin{proof}
By Theorem \ref{thm:disjointifying_closure_implies_involve} we have (1) implies (2), while (2) implies (3) is by Lemma \ref{lem:clmin_closure2}.
The definition of $\clmin$ gives (3) implies (4).
The fact that (4) implies (5) follows by an easy induction on $\alpha$ that if $\Krk(a, B) \le \alpha$ then $\Krk(a, B) < \omega_1$.
Finally (5) implies (1) will follow by the next section.
\end{proof}

There is likely a direct way to prove that (5) implies (1) without going through the results in the following section.

\section{Indiscernible support functions}

We define another technical notion which we will see is equivalent to an automorphism group involving $S_\infty$, with two goals in mind: to provide a motivating example of where a nontrivial disjointifying closure operator arises, and as a technical tool which we will make use of later.

Let $I$ be a set with an action $P \curvearrowright I$.
We write $[I]^{<\omega}$ to represent the set of finite subsets of $I$.
For $u, v, v'$ in $[\omega]^{<\omega}$, which represents the family of finite subsets of $\omega$, we write $v \isom_u v'$ iff there is some permutation $\sigma$ of $\omega$ such that $\sigma[v] = v'$ and $\sigma(n) = n$ for every $n \in u$.
Equivalently, $v \isom_u v'$ iff $|v| = |v'|$ and $v \cap u = v' \cap u$. Usually, we have $u \subseteq v$ and $u \subseteq v'$.

A function $\supp : [I]^{<\omega} \rightarrow [\omega]^{<\omega}$ is a \textbf{support function} iff
\begin{enumerate}
    \item[(1)] for every finite $A, B \subseteq I$ with $A \subseteq B$, we have $\supp(A) \subseteq \supp(B)$.
\end{enumerate}
We say $\supp$ is \textbf{nontrivial} iff furthermore
\begin{enumerate}
    \item[(2)] $\supp(A) \subsetneq \supp(B)$ for some finite $A \subsetneq B \subseteq I$;
\end{enumerate}
and finally we say $\supp$ is \textbf{indiscernible} iff
\begin{enumerate}
    \item[(3)] for every finite $A, B \subseteq I$ with $A \subseteq B$, and for every finite $u, v \subseteq \omega$ with $\supp(A) = u$ and $\supp(B) = v$, and for every $v' \isom_u v$, there exists some $B' \isom_A B$ with $\supp(B') = v'$.
\end{enumerate}

Note that $\supp$ being indiscernible implies that for any finite $A, B \subseteq I$ with $A \subseteq B$ there is some $B' \isom_A B$ such that $\supp(B') \cap \supp(B) = \supp(A)$.

Note that we do not make any demands that $\supp$ is invariant.
One could view this as meaning that a support function captures local information.
From the existence of such a function $\supp$, we will derive the existence of a nontrivial disjointifying closure operator.
An invariant closure operator, on the other hand, is a global object, as it describes relationships between sets which is invariant under automorphisms.

We assume that $\supp$ is such a function, and our objective is to show that $\clmin$ is nontrivial.

\begin{lemma}\label{lem:clmin_implies_supp}
Suppose $a \in \clmin(B)$. Then $\supp(aB) = \supp(B)$.
\end{lemma}

\begin{proof}
We prove by induction on $\alpha$ that if $\Krk(a, B) \le \alpha$ then $\supp(aB) = \supp(B)$.

Consider the case where $\alpha = 0$.
By indiscernibility of $\supp$, there is some $a' \isom_B a$ such that $\supp(a'B) \cap \supp(aB) = \supp(B)$.
By the definition of $\Krk(a,B) \le 0$, we know that $a' = a$, which means $\supp(a'B) = \supp(aB) = \supp(B)$ as desired.

Otherwise let $\alpha > 0$ and suppose the claim is true below $\alpha$.

In the first case of the definition of $\Krk(a, B) \le \alpha$, there is some $c \in I$ such that $\Krk(a, c'B) < \alpha$ for every $c' \isom_B c$.
By indiscernibility of $\supp$ fix some $c' \isom_B c$ such that $\supp(c'B) \cap \supp(aB) = \supp(B)$. 
By the induction hypothesis we have $\supp(ac'B) = \supp(c'B)$ and thus $\supp(aB) \subseteq \supp(c'B)$.
We conclude $\supp(aB) = \supp(B)$.

The second case of the definition of $\Krk(a, B) \le \alpha$ is handled a similar way.
Suppose for every $a' \isom_B a$, either $\Krk(a, a'B) < \alpha$ or $\Krk(a', aB) < \alpha$.
By indiscernibility of $\supp$, fix some $a' \isom_B a$ such that $\supp(a'B) \cap \supp(aB) = \supp(B)$.
In the case that $\Krk(a, a'B) < \alpha$ we have by the induction hypothesis that $\supp(aB) \subseteq \supp(aa'B) = \supp(a'B)$ and thus we can conclude $\supp(aB) = \supp(B)$.
On the other hand, if $\Krk(a', aB) < \alpha$ we have $\supp(a'B) \subseteq \supp(a'aB) = \supp(aB)$ and thus $\supp(aB) = \supp(B)$.
\end{proof}

Now we can connect indiscernible support functions with disjointifying closure operators.

\begin{proposition}\label{prop:supp_implies_cl}
If $P \curvearrowright I$ has a nontrivial indiscernible support function, then it has a nontrivial disjointifying closure operator.
\end{proposition}

\begin{proof}
Suppose $\supp$ is a nontrivial indiscernible support function.
It suffices to show that $\clmin$, which we showed is a disjointifying closure operator in Lemma \ref{lem:clmin_closure2}, is nontrivial.

Since $\supp$ is nontrivial, we can find some finite $A \subsetneq B$ such that $\supp(A) \subsetneq \supp(B)$.
In particular we can choose $A$ and $B$ such that $B = Ab'$ for some $b'$.
In particular, by Lemma \ref{lem:clmin_implies_supp} we have that $b' \not\in \clmin(A)$ and thus $\clmin$ is nontrivial.
\end{proof}

And we are done.

\subsection{Deriving an indiscernible support function from a Baire-measurable homomorphism}

Our final goal is to show that if $P \le S_I$ classifies $=^+$, then there is a nontrivial indiscernible support function on $P \curvearrowright I$.
We start by defining a presentation of $=^+$ which is easier for us to work with.

Let $\Delta \curvearrowright J$ be a free action of a countably-infinite group $\Delta$ on a countably-infinite set $J$ with infinitely-many orbits.
This partitions $J$ into infinitely-many ``copies'' of $\Delta$.
We assume $\Delta$ has the property that for finite $u, v \subseteq \Delta$ there is $\delta \in \Delta$ with $\delta u \cap v = \emptyset$, such as in the case of the additive group of the integers.
Let $T \subseteq J$ be a transversal for $\Delta \curvearrowright J$ (i.e. a set which intersects every $\Delta$-orbit exactly once) and fix an enumeration $T = \{t_n \mid n \in \omega \}$.
Fix also an enumeration $\Delta = \{\delta_m \mid m \in \omega\}$ where $\delta_0$ is the identity.
Thus every element of $J$ can be uniquely written in the form $\delta_m \cdot t_n$ for $m, n \in \omega$.

The group $S_J$ of permutations of $J$ has the natural action $S_J \curvearrowright J$.
Let $Q \le S_J$ be the closed subgroup of $q \in S_J$ that commute with the action of $\Delta$, i.e. for every $\delta \in \Delta$ we have $\delta \cdot (q \cdot a) = q \cdot (\delta \cdot a)$.

Now define $Y$ to be the $G_\delta$ set of all injections $y : J \rightarrow \mathbb{R}$.
This is a Polish space with the pointwise-convergence topology (putting the discrete topology on $J$), and moreover the natural action $Q \curvearrowright Y$ defined by $(q \cdot y)(a) = y(q^{-1} \cdot a)$ for all $q \in Q$, $y \in Y$, and $a \in J$ is continuous with respect to this topology.
This induces the orbit equivalence relation $E^Q_Y$.

\begin{lemma}
The equivalence relations $E^Q_Y$ and $=^+$ are Borel bi-reducible.
\end{lemma}

\begin{proof}
We first see that $=^+ \; \le_B E^Q_Y$.
Recall that $=^+$ is the equivalence relation defined on $\mathbb{R}^\omega$ where $(x_n) =^+ (y_n)$ iff $\{x_n \mid n \in \omega\} = \{y_n \mid n \in \omega\}$.
If we restrict $=^+$ to the invariant $G_\delta$ set $\mathbb{R}^\omega_{\inj}$ of injections, we get the equivalence relation $=^+_\inj$.
The two equivalence relations $=^+$ and $=^+_\inj$ are Borel bi-reducible, see e.g. \cite[10.3.4]{Gao2008}.
Thus it suffices to show $=^+_\inj \le_B E^Q_Y$.
The benefit to using $=^+_\inj$ is that it is induced by the natural action of $S_\infty$ on $\mathbb{R}^\omega_\inj$.

Let $g : \mathbb{R} \rightarrow \mathbb{R}^\omega$ be a Borel function satisfying 
\[\{g(x)(n) \mid n \in \omega\} \cap \{g(x')(n) \mid n \in \omega\} = \emptyset\]
for any $x \neq x' \in \mathbb{R}$.
In other words, $g$ sends distinct reals to enumerations of disjoint countably-infinite sets of reals.

Now define a Borel function $f : \mathbb{R}^\omega \rightarrow Y$ by
\[ f(p) = y_p \quad \text{where} \quad y_p(\delta_m \cdot t_n) = g(p(n))(m)\]
for every $t_n \in T$ and $\delta_m \in \Delta$.
We claim that this is a reduction from $=^+_\inj$ to $E^Q_Y$.
If $p, p' \in \mathbb{R}^\omega_\inj$ satisfy $p =^+ p'$ then there is some $\sigma \in S_\infty$ such that $p' = \sigma \cdot p$, or in other words $p'(n) = p(\sigma^{-1}(n))$ for every $n \in \omega$.
Fix the permutation $q \in Q$ where $q \cdot (\delta_m \cdot t_n) = \delta_m \cdot t_{\sigma(n)}$ for all $n$.
Then $q \cdot f(p) = q \cdot y_p$ and for every $m$ and $n$ we have
\begin{align*}
(q \cdot y_p)(\delta_m \cdot t_n) &= y_p(q^{-1}(\delta_m \cdot t_n)) & [\text{action of } Q \text{ on } Y]\\
&= y_p(\delta_m \cdot t_{\sigma^{-1}(n)}) & [\text{choice of } q]\\
&= g(p(\sigma^{-1}(n)))(m) & [\text{definition of } y_p]\\
&= g((\sigma \cdot p)(n))(m) & [\text{action of } S_\infty \text{ on } \mathbb{R}^\omega_\inj]\\
&= g(p'(n))(m) & [\text{choice of } \sigma]\\
&= y_{p'}(\delta_m \cdot t_n) & [\text{definition of } y_{p'}].
\end{align*}
as desired.
Conversely, if $p =^+ p'$ does not hold then the ranges of $y_p$ and $y_{p'}$ are unequal and so cannot be in the same $Q$-orbit since $q \cdot y_p$ will have the same range as $y_p$ for every $q \in Q$.

Next, we see that $E^Q_Y \le_B \; =^+$.
Fix a Borel bijection $s : \mathbb{R}^{\Delta} \rightarrow \mathbb{R}$.
Fix also a bijection $\langle \cdot, \cdot \rangle : \omega \times \omega \rightarrow \omega$ and for each $i, j \in \omega$ define $a_{\langle i, j\rangle} := \delta_i \cdot t_j$ so that we have $J = \{a_{\langle i, j\rangle} \mid i, j \in \omega\} = \{a_n \mid n \in \omega\}$.
Define a Borel function $f : Y \rightarrow \mathbb{R}^\omega$ by
\[ f(y) = x_y \quad \text{where} \quad x_y(\langle m, n \rangle) = s[(y(\delta \cdot \delta_m \cdot t_n))_{\delta \in \Delta}]. \]
We claim that $f$ is a reduction from $E^Q_Y$ to $=^+$.
Let $g \in Q$ and $y \in Y$.
We show that $f(g \cdot y) =^+ f(y)$.
Indeed,
\begin{align*}
\{f(g \cdot y)(n) \mid n \in \omega \} &= \{x_{g \cdot y}(n) \mid n \in \omega \} & [\text{definition of } f] \\
&= \{s[((g \cdot y)(\delta \cdot a_n))_{\delta \in \Delta}] \mid n \in \omega\} & [\text{definition of } x_{g \cdot y}]\\
&= \{s[(y(g^{-1} \cdot \delta \cdot a_n))_{\delta \in \Delta}] \mid n \in \omega\} & [\text{action of } Q \text{ on } Y]\\
&= \{s[(y(\delta \cdot g^{-1}\cdot a_n))_{\delta \in \Delta}] \mid n \in \omega\} & [\text{definition of } Q]\\
&= \{s[(y(\delta \cdot a))_{\delta \in \Delta}] \mid a \in J\} \\
&= \{s[(y(\delta \cdot a_n))_{\delta \in \Delta}] \mid n \in \omega\} \\
&= \{x_y(n) \mid n \in \omega\} & [\text{definition of } x_y] \\
&= \{f(y)(n) \mid n \in \omega\} & [\text{definition of } f].
\end{align*}

On the other hand, suppose $y, y' \in Y$ satisfy $f(y) =^+ f(y')$.
For each $i$ let $\langle m_i, n_i \rangle \in \omega$ such that $x_y(\langle 0, i \rangle) = x_{y'}(\langle m_i, n_i \rangle)$ and thus $y(\delta \cdot t_i) = y'(\delta \cdot \delta_{m_i} \cdot t_{n_i})$ for every $\delta \in \Delta$.
Since $y, y'$ are injections, such $\langle m_i, n_i\rangle$ are unique.
Let $q$ be the element of $Q$ satisfying $q \cdot t_i = \delta_{m_i} \cdot t_{n_i}$ for every $i \in \omega$.
We claim that $q \cdot y = y'$.
Indeed,
\begin{align*}
(q^{-1} \cdot y')(\delta \cdot t_i) &= y'(q \cdot \delta \cdot t_i) & [\text{action of } Q \text{ on } Y] \\
&= y'(\delta \cdot q \cdot t_i) & [\text{definition of } Q]\\
&= y'(\delta \cdot \delta_{m_i} \cdot t_{n_i}) & [\text{choice of } q] \\
&= y(\delta \cdot t_i) & [\text{choice of } \langle m_i, n_i \rangle]
\end{align*}
We are done.
\end{proof}

Recall that for $Y$ which is $G_\delta$ in the product space $\mathbb{R}^J$, the usual basis for the topology consists of relative products $Y \cap U$ where $U = \prod_{a \in J} U_a$ such that for some finite $u \subseteq J$ we have $U_a = \mathbb{R}$ if and only if $a \not\in u$.
We call $u$ the \emph{support} of $U$, and we call such $U$ the basic open sets.

Let $N \trianglelefteq Q$ be the closed normal subgroup of $n \in Q$ satisfying $n \cdot a \in \Delta \cdot a$ for every $a \in J$.
Let $\theta : \Delta^T \isom N$ be the natural isomorphism, where $\Delta^T$ is the full product $\prod_{t \in T} \Delta_t$ and each $\Delta_t$ is a copy of $\Delta$.
Given $n \in N$ we define the \emph{support} of $n$ to be the set of $t \in T$ such that $n \cdot t \neq t$.

Let $H \le Q$ be the closed subgroup of $h \in Q$ such that $h \cdot t \in T$ for all $t \in T$.
Let $\chi : S_T \isom H$ be the natural isomorphism, where $S_T$ is the permutation group of $T$.
Given $h \in H$ we define the \emph{support} of $h$ to be the set of $t \in T$ such that $h \cdot t \neq t$.

Then $Q = HN$ and $H \cap N = \{e\}$ and so $Q$ is the semidirect product of $H$ and $N$.
Given $g \in Q$, we say that the \emph{support} of $g$ is the set of $t \in T$ such that $g \cdot t \neq t$.
Easily, the support of $g$ is the union of the supports of $h$ and $n$ for the unique $h \in H$ and $n \in N$ with $g = hn$.
In particular, observing that for any $u \subseteq T$ and $g \in Q$ that $g \in \Stab_Q(u)$ if and only if the support of $g$ is disjoint from $u$, we see that any $g \in \Stab_Q(u)$ can be written as $g = hn$ for $h \in \Stab_H(u)$ and $n \in \Stab_N(u)$.

Let $Q_0$ be the set of all $q \in Q$ with finite support.
Then $Q_0$ is countable dense subgroup of $Q$.
Define similarly $H_0$ to be the $h \in H$ with finite support and $N_0$ to be the $n \in N$ with finite support, in which case we get $Q_0 = H_0N_0$.
For any $g \in \Stab_{Q_0}(u)$ we can write $g = hn$ where $h \in \Stab_{H_0}(u)$ and $n \in \Stab_{N_0}(u)$.
Define also $S_T^{\fin}$ to be the finite-support permutations of $T$ and so $\chi$ restricts to an isomorphism $S_T^\fin \isom H_0$.

Let $I$ be a countably-infinite set and let $P \le S_I$ be a closed subgroup with the natural action $P \curvearrowright I$.
Let $X$ be a Polish $P$-space.
Fix a Baire-measurable homomorphism $f : Y \rightarrow X$ from $E^Q_Y$ to $E^P_X$.
From $f$ we will derive an indiscernible support function.
Assuming that $f$ is not degenerate in a way which we will soon define, the support function will be non-trivial.

\begin{claim}\label{claim:density}
The following are all true:
\begin{enumerate}
\item if basic open $U, V \subseteq Y$ have disjoint support then $U \cap V \neq \emptyset$;
\item if $u$ is the support of $U$ and $g \in Q$ then $g \cdot u$ is the support of $g \cdot U$;
\item for any nonempty open $U \subseteq Y$ and finite $u, W \subseteq T$ such that the support of $U$ is contained in $\Delta \cdot u$ and $u \subseteq W$, the set of $y$ for which there is an involution $h \in H_0$ with support $u \cup (h \cdot u)$ such that $(h \cdot u) \cap W = \emptyset$ and $h \cdot y \in U$ is dense and open;
\item for any nonempty open $U \subseteq Y$, the set of $y \in Y$ such that $N_0 \cdot y \cap U \neq \emptyset$ is dense and open.
\end{enumerate}
\end{claim}

\begin{proof}
Statements (1) and (2) are easy.

For (3), the fact that the set of such $y$ is open is easy, so we proceed to showing that it is dense.

Let $V \subseteq Y$ be any nonempty open set and $v \subseteq T$ finite such that the support of $V$ is contained in $\Delta \cdot v$.
Choose $\sigma \in S_T^\fin$ to be any involution with support $u \cup \sigma[u]$ where $\sigma[u] \cap W = \emptyset$ and $\sigma[u] \cap v = \emptyset$.
Let $h := \chi(\sigma)$ which is an involution and has support as desired.
By Claim (2) the support of $h \cdot U$ is contained in $\Delta \cdot (h \cdot u) = \Delta \cdot \sigma[u]$.
Then by Claim (1) we have $(h \cdot U) \cap V \neq \emptyset$.
Fixing any $y \in (h \cdot U) \cap V$ we see that $h^{-1} \cdot y \in U$.

For (4), let $V \subseteq Y$ be any nonempty open set.
Suppose $u \subseteq J$ is the support of $U$ and $v \subseteq J$ is the support of $V$.
By choice of $\Delta$ we can find $n \in N_0$ such that $(n \cdot u) \cap v = \emptyset$.
Indeed, let $\Delta_u$ be the finite set of all $\delta$ such that $\delta \cdot t \in u$ for some $t \in T$, and $\Delta_v$ the finite set of all $\delta$ such that $\delta \cdot t \in v$ for some $t \in T$ and choose $\hat{\delta}$ such that $(\hat{\delta} \cdot \Delta_u) \cap \Delta_v = \emptyset$.
Then let $n \in N_0$ such that $n \cdot t = \hat{\delta} \cdot t$ for every $t \in T \cap \Delta \cdot u$.
By (2), $n \cdot u$ is the support of $n \cdot U$ and by (1) there is some $y \in V \cap n \cdot U$.
Then $n^{-1} \cdot y \in U$.
\end{proof}

For each basic open $U$ and finite $W \subseteq T$, let $D^0_{U, W}$ be the set of $y$ as in Claim \ref{claim:density}.(3) and $D^1_{U}$ the set of $y$ as in Claim \ref{claim:density}.(4).

Fix $C_0 \subseteq Y$ comeager satisfying clauses (1)-(5) of Lemma \ref{lem:orbit_continuity}.
Let $C$ be the intersection of $C_0$ with all $g_0 \cdot D^0_{U, W}$ and $g_0 \cdot D^1_{U}$ for $g_0 \in Q_0$, basic open $U$, and finite $W \subseteq T$.
Since each $D^0_{U, W}$ and $D^1_{U}$ is open and dense by Claim \ref{claim:density}, $C$ is a countable intersection of comeager sets and thus comeager.
We claim that $C$ satisfies:
\begin{enumerate}
\item $f$ is continuous on $C$;
\item for all $y \in C$, $\forall^* g \in Q, \; g \cdot y \in C$;
\item for all $y \in C$, $\forall g \in Q_0, \; g \cdot y \in C$;
\item for all $y_0 \in C$ for every open neighborhood $V \subseteq P$ of the identity there is open $U \ni y_0$ and open neighborhood $W \subseteq Q$ of the identity such that for all $y \in C \cap U$ and $w \in W$ if $w \cdot y \in C \cap U$ then $f(w \cdot y) \in V \cdot f(y)$;
\item for all $y_0 \in C$ and $g_0 \in Q_0$ and nonempty open $W \subseteq P$ if $f(g_0 \cdot y_0) \in W \cdot f(y_0)$ then there is open $U \ni y_0$ such that for all $y \in C \cap U$, $g_0 \cdot y \in C$ and $f(g_0 \cdot y) \in W \cdot f(y)$.
\end{enumerate}
Indeed, clauses (1), (4), and (5) follow from the fact that $C \subseteq C_0$ and clause (3) is easily verified since $C_0$ is $Q_0$-invariant and so is the intersection of all the $g_0 \cdot D^0_{U, W}$ and $g_0 \cdot D^1_{U}$.
For clause (2), let $y \in C$. For any $g_0 \in Q_0$, basic open $U$, and finite $W$, the set of $g \in Q$ such that $g \cdot y$ is in $g_0 \cdot D^0_{U, W}$ and $g_0 \cdot D^1_U$ is open by continuity of the action and contains $Q_0$ by definition of $C$. Since $Q_0$ is dense in $Q$, this set is open and dense, hence comeager. The intersection of these countably-many sets along with the comeager set from clause (2) for $C_0$ is still comeager.

Fix some point $y_0 \in C$.
For finite $A \subseteq I$ and $u \subseteq T$, say that $u$ \emph{supports} $A$ iff there is a basic open neighborhood $U \ni y_0$ such that for every $y \in U \cap C$ and for every $g \in \Stab_{Q_0}(u)$, if $g \cdot y \in U$ then $f(g \cdot y) \in \Stab_P(A) \cdot f(y)$.
By properties (3) and (4) of $C$, for every finite $A$ there exists some finite $u$ such that $u$ supports $A$.

We would like to define the \emph{support} of any such $A$, but this requires a fair amount of work, namely the next few claims.

Our proof skids to a halt while we prove a technical lemma about permutations and involutions.
Recall that $S^\fin_T$ represents the countable dense subgroup of permutations of $T$ with finite support, meaning all but finitely-many elements are fixed.

\begin{claim}\label{claim:involution}
Let $u, v \subseteq T$ be finite and $\pi \in \Stab_{S^\fin_T}(u \cap v)$, and let $\sigma \in \Stab_{S^\fin_T}(v)$ be an involution with support $(u \setminus v) \cup \sigma[u \setminus v]$ such that $\sigma[u \setminus v]$ is disjoint from $u$, $v$, and the support of $\pi$.
Then there is some $\pi_1 \in \Stab_{S^\fin_T}(u)$ such that $\pi = \sigma \circ \pi_1 \circ \sigma$.
\end{claim}

\begin{proof}

We define $\pi_1 \in S_T$ by
\begin{equation*}
\pi_1(a) = 
    \begin{cases}
        \sigma(\pi(\sigma(a))) & \quad (A) \; \text{if } a \in \sigma[u \setminus v] \; \text{and} \; \pi(\sigma(a)) \in u \setminus v\\
        \pi(\sigma(a)) & \quad (B) \; \text{if } a \in \sigma[u \setminus v] \; \text{and} \; \pi(\sigma(a)) \not\in u \setminus v\\
        \sigma(\pi(\sigma(a))) & \quad (C) \; \text{if } a\not\in u \setminus v \; \text{and} \; \pi(a) \in u \setminus v \\
        a & \quad (D) \; \text{if } a \in u\\
        \pi(a) & \quad (E) \; \text{otherwise}.
    \end{cases}
\end{equation*}
Of course, we must show this is well-defined.
Conditions (A) and (B) are mutually exclusive by definition.
Furthermore, if $a$ falls into either case (A) or (B) then we have $a \in \sigma[u \setminus v]$, but by choice of $\sigma$ we would have $a$ not in the support of $\pi$, which would imply that $\pi(a) = a$ but that would make satisfying case (C) impossible.
Case (D) cannot coincide with either case (A) or (B) because if $a \in u$ then $a \not\in \sigma[u]$.
Case (D) cannot coincide with case (C) because if $a \in u$ and $a \not\in u \setminus v$ then this would imply $a \in u \cap v$, but since $\pi \in \Stab(u \cap v)$ this would mean $\pi(a) = a$ but then $\pi(a) \in u \setminus v$ cannot hold.

To see that $\pi = \sigma \circ \pi_1 \circ \sigma$, we fix an arbitrary $a \in T$ and check that $\pi(a) = (\sigma \circ \pi_1 \circ \sigma)(a)$. 

We case on which condition of $\pi_1$ that $\sigma(a)$ lands.

If $\sigma(a)$ satisfies condition (A) or (C) then by definition of $\pi_1$ we have $\pi_1(\sigma(a)) = \sigma(\pi(\sigma(\sigma(a))))$ and thus $\sigma(\pi_1(\sigma(a))) = \sigma(\sigma(\pi(\sigma(\sigma(a))))) = \pi(a)$ since $\sigma$ is an involution.

Next suppose $\sigma(a)$ satisfies condition (B), or equivalently, since $\sigma$ is an involution, we have $a \in u \setminus v$ and $\pi(a) \not\in u \setminus v$.
In particular, $a$ is in the support of $\pi$ and so by choice of $\sigma$ we have $\pi(a) \not\in \sigma[u \setminus v]$.
In particular, $\pi(a)$ is not in $(u \setminus v) \cup \sigma[u \setminus v]$ which is the support of $\sigma$, and so $\sigma(\pi(a)) = \pi(a)$.
By definition of $\pi_1$ we have $\pi_1(\sigma(a)) = \pi(\sigma(\sigma(a))) = \pi(a)$ and thus $\sigma(\pi_1(\sigma(a))) = \sigma(\pi(a)) = \pi(a)$.

Now suppose $\sigma(a)$ satisfies condition (D), i.e. $\sigma(a) \in u$.
Then by definition of $\pi_1$ we have $\sigma(\pi_1(\sigma(a))) = \sigma(\sigma(a)) = a$.
This is equal to $\pi(a)$.
Indeed, if $\sigma(a) \in u \cap v$ then $a = \pi(a)$ since $\pi$ is in the stabilizer of $u \cap v$.
On the other hand if $\sigma(a) \in u \setminus v$ then $a \in \sigma[u \setminus v]$ and in particular is not in the support of $\pi$ so $\pi(a) = a$.

Finally suppose $\sigma(a)$ satisfies condition (E).
Then by definition of $\pi_1$ we have $\sigma(\pi_1(\sigma(a))) = \sigma(\pi(\sigma(a)))$.
Since $\sigma(a)$ does not satisfy conditions (A) or (B), we have $\sigma(a) \not \in \sigma[u \setminus v]$, or equivalently, $a \not\in u \setminus v$.
Moreover $\sigma(a) \not\in u$ since $\sigma(a)$ does not satisfy condition (D).
Thus $a$ is not in the support of $\sigma$ and so $\sigma(a) = a$.
We also have $\sigma(\pi(a)) = \pi(a)$ because if $a$ is not in the support of $\pi$ this is clearly true, and otherwise if $\pi(a)$ is in the support of $\sigma$ it would have to be in $u \setminus v$, but this would imply that $\sigma(a) = a$ satisfies condition (C).
\end{proof}

Having finished this unpleasant detour, we continue with our proof.

\begin{claim}
For every finite $A \subseteq I$ and $u, v \subseteq T$, if both $u$ and $v$ support $A$ then $u \cap v$ supports $A$.
\end{claim}

\begin{proof}
By taking an intersection, let $U \ni y_0$ be a basic open set witnessing both that $u$ and $v$ support $A$.
We argue that $U$ in fact witnesses that $u \cap v$ supports $A$ as well.
To that end, let $y \in U \cap C$ and $g \in \Stab_{Q_0}(u \cap v)$ be arbitrary such that $g \cdot y \in U$.
We need to show that $f(g \cdot y) \in \Stab_P(A) \cdot f(y)$.
We will do this by finding $g_1, g_2, g_3 \in Q_0$ such that 
\begin{enumerate}
\item[(i)] $g = g_3g_2g_1$
\item[(ii)] $g_1, g_3 \in \Stab_{Q_0}(v)$ and $g_2 \in \Stab_{Q_0}(u)$
\item[(iii)] $g_1 \cdot y$ and $g_2g_1 \cdot y$ and $g_3g_2g_1\cdot y$ are in $C \cap U$.
\end{enumerate}
This will suffice as this would imply that $f(g_1 \cdot y)
 \in \Stab_P(A) \cdot f(y)$ and $f(g_2g_1 \cdot y) \in \Stab_P(A) \cdot f(g_1 \cdot y)$ and $f(g \cdot y) = f(g_3g_2g_1 \cdot y) \in \Stab_P(A) \cdot f(g_2g_1 \cdot y)$ which all together implies $f(g \cdot y) \in \Stab_P(A) \cdot f(y)$.

Write $g = hn$ where $h \in \Stab_{H_0}(u \cap v)$ and $n \in \Stab_{N_0}(u \cap v)$.
Write $h = \chi(\pi)$ where $\pi \in \Stab_{S_T^\fin}(u \cap v)$.
Write $n =  n_{v \setminus u} \hat{n} n_{u \setminus v}$ where the support of $\hat{n}$ is disjoint from $(u \cup v)$ and $\supp(n_{v \setminus u}) \subseteq v \setminus u$ and $\supp(n_{u \setminus v}) \subseteq u \setminus v$.

Let $W \subseteq T$ be a large enough finite set such that it contains the support of $\pi$ and $n$ and $u$ and $v$ and also $\Delta \cdot W$ contains the support of $U$.

By property (3) of $C$, since $n_{u \setminus v} \cdot y \in D^0_{U, W}$, we find a finite-support involution $\sigma \in S^\fin_T$ with support $(u \setminus v) \cup \sigma(u \setminus v)$ such that $\sigma(u \setminus v) \cap W = \emptyset$ and $\chi(\sigma) n_{u \setminus v} \cdot y \in U$.

Let $\pi_1 \in \Stab_{S_T^\fin}(u)$ as in Claim \ref{claim:involution} so that $\pi = \sigma \circ \pi_1 \circ \sigma$.

Now define $g_1 = \chi(\sigma) n_{u \setminus v}$ and $g_2 = \chi(\pi_1) n_{v \setminus u} \hat{n}$ and $g_3 = \chi(\sigma)$.
We can easily see that condition (ii) is satisfied.
Moreover, observing that $\chi(\sigma)$ commutes with both $\hat{n}$ and $n_{v \setminus u}$ as they have disjoint supports, we indeed have
\begin{align*}
g_3g_2g_1 &= \chi(\sigma)\chi(\pi_1) n_{v \setminus u}\hat{n} \chi(\sigma) n_{u \setminus v} = \chi(\sigma)\chi(\pi_1)  \chi(\sigma) n_{v \setminus u}\hat{n} n_{u \setminus v}\\
&= \chi(\sigma \circ \pi_1 \circ \sigma) n = \chi(\pi) n = h n
\end{align*}
and so condition (i) is met.

Moving on to condition (iii), the only thing left to check is that $g_2g_1 \cdot y \in U$.
Write $U = U_{u \setminus v} \cap \hat{U}$ where $U_{u \setminus v}$ has support contained in $\Delta \cdot (u \setminus v)$ and $\hat{U}$ has support disjoint from $\Delta \cdot (u \setminus v)$.
Since $\chi(\sigma) n_{u \setminus v} \cdot y \in U \subseteq U_{u \setminus v}$ and $\chi(\pi_1) n_{v \setminus u} \hat{n} \in \Stab_{Q_0}(u)$ we have $g_2g_1 \cdot y \in U_{u \setminus v}$.
On the other hand, since $\chi(\pi) n \cdot y \in U \subseteq \hat{U}$ and $\sigma \in \Stab_{S_T}(W \setminus (u \setminus v))$ we have $g_2g_1 \cdot y =g_3^{-1}g \cdot y =  \chi(\sigma) \chi(\pi) n \cdot y \in \hat{U}$.
Thus $g_2g_1 \cdot y \in U_{u \setminus v} \cap \hat{U} = U$.

We are done.
\end{proof}

We conclude in particular that for every $A$ there is a minimal $u$ which supports $A$.
The minimal support of $A$, denoted $\supp(A)$, is \textbf{the support of $A$}.
Now we wish to show that $\supp$ is indiscernible.
We will prove this after the following technical claim.

\begin{claim}\label{claim:supp_inv}
For any $u$ and $A$, if $u$ supports $A$ as witnessed by neighborhood $U_0 \ni y_0$, then for any finite-support $\sigma \in S_T$ and $n \in N_0$ and $h \in P$ with $\chi(\sigma)n \cdot y_0 \in U_0$ and $f(\chi(\sigma)n \cdot y_0) = h \cdot f(y_0)$, we have that $\sigma^{-1} \cdot u$ supports $h^{-1} \cdot A$.
\end{claim}

\begin{proof}
Let $v = \sigma^{-1} \cdot u$ and $B = h^{-1} \cdot A$. 
Observe that $\Stab_P(B) = h^{-1} \Stab_P(A) h$ and $\Stab_{S_T}(v) = \sigma^{-1} \Stab_{S_T}(u) \sigma$. 
We show that $v$ supports $B$. 

By Clause (5) of the choice of $C$, we may find open $U_1 \ni y_0$ such that for every $y \in U_1 \cap C$, $f(\chi(\sigma)n \cdot y) \in h \Stab_P(B) \cdot f(y)$.
By shrinking $U_1$ if necessary we can also guarantee for any such $y$ that $\chi(\sigma)n \cdot y \in U_0$.
We argue that $U_1$ witnesses that $v$ supports $B$.

Let $y \in U_1 \cap C$ be arbitrary. 
Let $g_v \in \Stab_{Q_0}(v)$ arbitrary and suppose $g_v \cdot y \in U_1$. 
We want to show that $f(g_v \cdot y) \in \Stab_P(B) \cdot f(y)$.
Indeed, because $y \in U_1$ we have
\begin{equation}\label{eq:indis1}
    f(\chi(\sigma)n \cdot y) \in h \Stab_P(B) \cdot f(y)
\end{equation}
and because $g_v \cdot y \in U_1$ we have
\begin{equation}\label{eq:indis2}
    f(\chi(\sigma)n g_v \cdot y) \in h \Stab_P(B) \cdot f(g_v \cdot y).
\end{equation}
Observe that $ng_vn^{-1} \in \Stab_{Q_0}(v) = \chi(\sigma)^{-1} \Stab_{Q_0}(u) \chi(\sigma)$ so we have
\begin{equation}\label{eq:indis25}
(\chi(\sigma)n) g_v (\chi(\sigma)n)^{-1} \in \Stab_{Q_0}(u).
\end{equation}

Therefore, because $\chi(\sigma)n \cdot y \in U_0$ and $\chi(\sigma)ng_v \cdot y \in U_0$ and 
\[\chi(\sigma)ng_v \cdot y = [(\chi(\sigma)n) g_v (\chi(\sigma)n)^{-1}] \chi(\sigma)n \cdot y,\]
by Equation \ref{eq:indis25} we have
\begin{equation}\label{eq:indis3}
    f(\chi(\sigma)n \cdot y) \in \Stab_P(A) \cdot f(\chi(\sigma)n g_v \cdot y).
\end{equation}
Thus putting these together we have
\begin{align*}
    f(g_v \cdot y) & \in \Stab_P(B) h^{-1} \cdot f(\chi(\sigma)n g_v \cdot y)\\
    & \subseteq \Stab_P(B) h^{-1} \Stab_P(A) \cdot f(\chi(\sigma)n \cdot y) \\
    & \subseteq \Stab_P(B) h^{-1} \Stab_P(A) h \Stab_P(B) \cdot f(y) \\
    & = \Stab_P(B) \cdot f(y),
\end{align*}
where the first inclusion is from Equation \ref{eq:indis2}, the second from Equation \ref{eq:indis3}, and the third from Equation \ref{eq:indis1} as desired.
\end{proof}

Now we see that $\supp$ is indiscernible.
Let $A, B \subseteq I$ be finite and $u, v \subseteq T$ finite such that $\supp(A) = u$ and $\supp(B) = v$, and let $U_0 \ni y_0$ be a basic open neighborhood witnessing this.
Let $v' \subseteq T$ be arbitrary such that $v' \isom_u v$.

Fix a finite-support $\sigma \in S_T$ such that $\sigma(a) = a$ for every $a \in u$ and $\sigma[v'] = v$.
We claim that there is some $n \in \Stab_{N_0}(u)$ such that $\chi(\sigma) n \cdot y_0 \in U_0$.
Indeed, write $U_0 = V_0 \cap V_1$ where the support of $V_0$ is contained in $\Delta \cdot u$ and the support of $V_1$ is disjoint from $\Delta \cdot u$.
Since $y_0 \in C \subseteq D^1_{\chi(\sigma)^{-1} \cdot V_1}$ we can find some $\hat{n} \in N_0$ such that $\hat{n} \cdot y_0 \in \chi(\sigma)^{-1} \cdot V_1$.
Write $\hat{n} = n_0 n_1$ where the support of $n_0$ is contained in $u$ and the support of $n_1$ is disjoint from $u$.
We claim that $n := n_1$ works.
Indeed, $\chi(\sigma) n_1 \cdot y_0 \in V_0$ because $\chi(\sigma) n_1 \in \Stab_{Q_0}(u)$ and so $(\chi(\sigma)n_1)^{-1} \cdot V_0 = V_0$.
And $\chi(\sigma)n_1 \cdot y_0 \in V_1$ because $\chi(\sigma)n_0n_1 \cdot y_0 =\chi(\sigma) \hat{n} \cdot y_0 \in V_1$ and $\chi(\sigma)$ and $n_0$ commute and $n_0^{-1} \cdot V_1 = V_1$.

Because $\chi(\sigma)n \in \Stab_{Q_0}(u)$, there is some $g \in \Stab_P(A)$ such that $f(\chi(\sigma)n \cdot y_0) = g \cdot f(y_0)$.
By Claim \ref{claim:supp_inv} this means that $v'$ is a support of $g^{-1} \cdot B$, and because $g \in \Stab_P(A)$, we have $B \isom_A g^{-1} \cdot B$.
A symmetric argument gives us that $v'$ is in fact the minimal support of $g^{-1} \cdot B$, as desired.

Our final task is to show that $\supp$ is nontrivial.
\begin{claim}
If $\supp$ is trivial, then $f(y) \mathrel{E^P_X} f(y')$ for every $y, y' \in C$.
\end{claim}

\begin{proof}
Suppose $\supp$ is trivial, i.e. the support of every $A$ is $u$ for some fixed $u$.
In other words, for every finite $A \subseteq I$, there is an open neighborhood $U_A$ of $y_0$ such that for every $y \in U_A \cap C$ and every $g \in \Stab_{Q_0}(u)$ with $g \cdot y \in U_A$, we have $f(g \cdot y) \in \Stab_P(A) \cdot f(y)$.

Let $y \in Y \cap C$ be arbitrary.
We will show that $f(y) \mathrel{E^P_X} f(y_0)$.
First we claim that for every basic open $U$ of $y_0$ there is some $n \in \Stab_{N_0}(u)$ such that $n \cdot y \in U$.
Indeed, since $y \in D^1_{U}$ there is some $\hat{n} \in N_0$ such that $\hat{n} \cdot y \in U$.
Write $\hat{n} = n_0n_1$ for $n_0, n_1 \in N_0$ where the support of $n_0$ is contained in $u$ and the support of $n_1$ is disjoint from $u$, and writing $U = V_0 \cap V_1$ where the support of $V_0$ is contained in $\Delta \cdot u$ and the support of $V_1$ is disjoint from $\Delta \cdot u$.
We claim that $n_1 \cdot y \in U$.
Indeed, $n_1 \cdot y \in V_0$ because $n_0n_1 \cdot y \in U \subseteq V_0$ and $n_0^{-1} \cdot V_0 = V_0$, and $n_1 \cdot y \in V_1$ because $y \in U \subseteq V_1$ and $n_1 \cdot V_1 = V_1$.
Observe that $n_1 \in \Stab_{N_0}(u)$.

Write $I$ as an increasing union $\bigcup_n A_n$ of finite sets.
Fix a compatible complete metric $d$ on $Y$.
For each $n$, let $U_n$ be a basic open neighborhood of $y_0$ with $d$-diameter less than $1/n$ and contained in $U_{A_n}$.
Let $g_0$ be the identity and for every $n$, let $g_{n+1} \in Q_0$ (actually in $N_0$) such that $g_{n+1} \cdot y \in C \cap U_{n}$.
We have $g_n \cdot y \rightarrow y_0$, and so by continuity of $f$ on $C$, we have $f(g_n \cdot y) \rightarrow f(y_0)$.

For every $n$ we may find $h_n \in \Stab_P(A_n)$ such that $f(g_{n+1}\cdot y) = h_n \cdot f(g_n \cdot y)$.
Defining $h^*_n := h_n ... h_0$, we have $f(g_{n+1} \cdot y) = h^*_n \cdot f(y)$ for every $n$.
We claim that the sequence $(h^*_n)$ is Cauchy and thus $h^*_n \rightarrow h^*_\infty$ for some $h^*_\infty \in P$, and so by continuity of the group action, we have $h^*_\infty \cdot f(y) = f(y_0)$.

Indeed, let $I = \{b_i \mid i \in \omega\}$ and recall for $h, h' \in S_I$ if we define $d(h, h') := 2^{-n-1}$ where $n$ is least such that $h(b_n) \neq h'(b_n)$ and if $D(h, h') := d(h, h') + d(h^{-1}, h'^{-1})$ then $D$ is a compatible complete metric on $S_I$ and thus on the closed subgroup $P$.
The sequence $(h^*_n)$ is Cauchy because for any $n$, choosing $m$ big enough such that $A_m \supseteq \{b_i \mid i < n\}$ we have for any $m < m_0 < m_1$ that $h^*_{m_0} \upharpoonright A_m = h^*_{m_1} \upharpoonright A_m$ and $(h^*_{m_0})^{-1} \upharpoonright A_m = (h^*_{m_1})^{-1} \upharpoonright A_m$ and so $D(h^*_{m_0}, h^*_{m_1}) \le 2^{-n-1} + 2^{-n-1} = 2^{-n}$.
\end{proof}

Observe that for any $a, a' \in J$, the set of $y$ such that $y(a) \neq y_0(a')$ is dense and open.
Thus there is some $y \in C$ such that $y(a) \neq y_0(a')$ for every $a, a' \in J$.
In particular, this means $Q \cdot y \neq Q \cdot y_0$.
By the same argument, every orbit of $E^Q_Y$ is meager.
We conclude that if $f$ is not just a homomorphism but a reduction, then $\supp$ is nontrivial.
The same could be concluded if $f$ witnesses that $E^Q_Y$ is not generically ergodic with respect to $E^P_X$.

Combining the results of this section, we get:

\begin{theorem}
Let $I$ be a countable set and $P \le S_I$ a closed subgroup with the natural action $P \curvearrowright I$.
The following are equivalent:
\begin{enumerate}
    \item $P$ classifies $=^+$;
    \item there is a Polish $P$-space $X$ such that $E^Q_{Y}$ is not generically-ergodic with respect to $E^P_X$;
    \item $P \curvearrowright I$ has a nontrivial indiscernible support function;
    \item $P \curvearrowright I$ has a nontrivial disjointifying closure operator.
\end{enumerate}
\end{theorem}

\begin{proof}
To see the fact that (1) implies (2), use the fact we just mentioned that all classes of $E^Q_{Y}$ are meager.
The fact that (2) implies (3) was the result of this last subsection, while (3) implies (4) was Proposition \ref{prop:supp_implies_cl}.

To see that (4) implies (1), we have by Corollary \ref{cor:equivalences} that if (4) holds then $P$ must involve $S_\infty$ so by Lemma \ref{lem:mackey}, $P$ classifies $=^+$ since $=^+$ is classifiable by $S_\infty$.
\end{proof}

We have now proved all of the equivalences of the main theorem.

\bibliographystyle{alpha}
\bibliography{main}
\end{document}